\documentclass[11pt,reqno]{amsart}
\usepackage[utf8]{inputenc}
\usepackage{amsmath,amsthm,amssymb}
\usepackage{graphics}
\usepackage{hyperref}
\usepackage[usenames, dvipsnames]{xcolor}
\usepackage{mathtools}
\mathtoolsset{showonlyrefs}
\usepackage[square,sort,comma,numbers]{natbib}
\definecolor{darkblue}{rgb}{0.0,0.0,0.3}
\hypersetup{colorlinks,breaklinks,
  linkcolor=darkblue,urlcolor=darkblue,
anchorcolor=darkblue,citecolor=darkblue}

\theoremstyle{plain}
\newtheorem{theorem}{Theorem}[section]
\newtheorem*{theorem*}{Theorem}
\newtheorem{lemma}[theorem]{Lemma}
\newtheorem{proposition}[theorem]{Proposition}
\newtheorem*{proposition*}{Proposition}
\newtheorem{corollary}[theorem]{Corollary}
\newtheorem*{corollary*}{Corollary}

\newtheorem{conjecture}[theorem]{Conjecture}
\newtheorem*{conjecture*}{Conjecture}

\theoremstyle{definition}
\newtheorem{remark}[theorem]{Remark}
\newtheorem*{remarks}{Extended Remarks}

\numberwithin{equation}{section}

\renewcommand{\Im}{\operatorname{Im}}
\renewcommand{\Re}{\operatorname{Re}}
\DeclareMathOperator{\SL}{SL}

\DeclareMathOperator*{\Res}{Res}
\DeclareMathOperator{\sgn}{sgn}

\makeatletter
\let\@wraptoccontribs\wraptoccontribs
\makeatother

\title{Counting Divisors in the Outputs of a Binary Quadratic Form}
\author[Kuan]{Chan Ieong Kuan}
\author[Lowry-Duda]{David Lowry-Duda}
\author[Walker]{Alexander Walker}
\contrib[with an appendix by]{Tinghao Huang and Chan Ieong Kuan}

\begin{document}

\maketitle

\begin{abstract}
For a fixed natural number $h$, we prove meromorphic continuation of the
two-variable Dirichlet series $\sum_m r_2(m) \sigma_w(m + h) (m + h)^{-s + w}$
to $\mathbb{C}^2$ and use this to obtain asymptotics for
$\sum_{m^2 + n^2 \leq X} \sigma_w(m^2 + n^2 + h)$.
We approach this continuation through spectral theory.
Our results are comparable to earlier work of Bykovskii, who used
different methods to study the sums $\sum_{n^2 \leq X} \sigma_w(n^2 + h)$.
\end{abstract}

\section{Introduction}

In~\cite{Hooley63}, Hooley studies the average number of divisors in a quadratic sequence,
\begin{equation}
  S^1(X) = \sum_{n \leq X} d(n^2 + h).
\end{equation}
If $-h = k^2$ is a perfect square, then $n^2 + h = (n-k)(n+k)$ and the summands
resemble $d(n-k)d(n+k)$.
As Hooley notes, it is possible to adapt Ingham's methods of studying partial
sums $\sum d(n) d(n + h)$ from~\cite{ingham1927some} to show that
\begin{equation}\label{eq:ingham_basic}
  S^1(X) = c_h X (\log X)^2 + O_h(X \log X) \qquad \qquad (h = -k^2).
\end{equation}
Focusing on the case when $-h$ is not a perfect square, Hooley shows that
\begin{equation}\label{eq:hooley_basic}
  S^1(X) = c_h X \log X + c_h' X + O_h(X^{\frac{8}{9}} \log^3 X) \qquad \qquad (h \neq -k^2)
\end{equation}
for suitable constants $c_h, c_h' \neq 0$, by manipulating exponential sums.

Bykovskii~\cite{Bykovskii87} considers more general bounds when $h > 0$ by
studying the spectral expansion of a particular automorphic form.
Let $\sigma_\nu(n) = \sum_{d \mid n} d^\nu$ denote the sum-of-divisors
function.
Then Bykovskii proved asymptotics of the form
\begin{equation}\label{eq:bykovskii_main}
  \sum_{n \leq X} \sigma_{-s}(n^2 + h)
  =
  c_{s}(h) X \log X
  +
  c_{s}'(h) X
  +
  c_{s}''(h) X^{1 - 2s}
  +
  \mathrm{Err}_{s}^1(X),
\end{equation}
where $0 \leq \Re s \leq 1$. The constants are given explicitly as sums of special values of
weight $0$ Eisenstein series, and the first constant $c_{s}(h)$ equals $0$ unless $s = 0$.
Several different forms of asymptotic are given, depending on the value of $\Re s$.
In particular, Bykovskii shows that
\begin{equation}\label{eq:bykovskii_error}
  \mathrm{Err}_{s}^1(X) \ll_{s,h}
  \begin{cases}
    X^{\frac{2}{3}} \log^{\frac{2}{3}} X, & \Re s = 0, \\
    X^{\frac{1}{\frac{3}{2} + \Re s}}, & 0 < \Re s < \frac{1}{2}.
  \end{cases}
\end{equation}

Note that the error term $O(X^{\frac{2}{3}} \log^{\frac{2}{3}} X)$ in the case $s = 0$
gives a substantial improvement over~\eqref{eq:hooley_basic}.
The primary obstruction leading to the error terms~\eqref{eq:bykovskii_error}
arises from the contributions of weight $0$ Maass forms corresponding to eigenvalues on
$\SL(2, \mathbb{Z})$.

Our aim is to study the related sums
\begin{equation}
\begin{split}
  S(X) = S^2(X; w, h)
  :=&
  \sum_{m^2 + n^2 + h \leq X} \sigma_{1 - 2w}(m^2 + n^2 + h) \\
  =& \sum_{m \leq X - h} r_2(m) \sigma_{1 - 2w}(m + h),
\end{split}
\end{equation}
where $r_k(n)$ denotes the number of representations of $n$ as a sum of $k$
squares.
This generalizes partial sums of the form $\sum d(n^2 + m^2 + h)$
and is an $r_2$ analog of the problem studied by Bykovskii.
Yet, despite the obvious similarities,
Bykovskii's method is specific to $n^2+h$ and does not directly extend to the $m^2+n^2+h$ analog.
A new approach is necessary.

Our primary result is the following theorem.

\begin{theorem}\label{thm:main_introduction}
  Fix $w \in \mathbb{C}$ with $0 < \Re w < 1$, a positive integer $h$, and any
  $\epsilon > 0$. If $w \neq \frac{1}{2}$, we have
  \begin{align*}
    S(X)
    &=
    (4\pi)^{\frac{1}{2}}
    \zeta^*(2w) \varphi_h(\tfrac{1}{2}+w)
    h^{\frac{1}{2}-w} X
    \\
    & +
    (4\pi)^{\frac{1}{2}}
    \zeta^*(2-2w) \varphi_h(\tfrac{3}{2} -w)
    \frac{h^{w-\frac{1}{2}} X^{2-2w}}
         {2 -2w}
		+ \mathrm{Err}_w^2(X),
  \end{align*}
	in which
\begin{equation}
	\mathrm{Err}_w^2(X) \ll_{h,w,\epsilon}
		\begin{cases}
			X^{\frac{3}{2\Re w + 3} + \epsilon}, & \Re w \in [\tfrac{1}{2}, 1), \\
			X^{2-2\Re w - \frac{2-2\Re w}{5-2\Re w} + \epsilon}, & \Re w \in (0,\tfrac{1}{2}],
		\end{cases}
\end{equation}
  where $\zeta^*$ denotes the completed Riemann $\zeta$ function and
  $\varphi_h(w)$ may be defined in terms of explicitly computable Fourier coefficients of
  Eisenstein series (as described further in
  Proposition~\ref{prop:regularized-series} and
  Lemma~\ref{lem:Eisenstein-series-coefficients-formulas-wt1}).
  If instead $w=\frac{1}{2}$, we have
  \begin{align*}
  S(X)
  =&
    (4 \pi)^{\frac{1}{2}}
    \varphi_h(1)
    X \log X
  +
    (4 \pi)^{\frac{1}{2}}
    \varphi_h(1)(\gamma-\log(4\pi h))
    X
  \\
  & \qquad +
    (4 \pi)^{\frac{1}{2}}
    \varphi_h'(1)
    X
  -
    (4 \pi)^{\frac{1}{2}}
    \varphi_h(1)
    X
  +
  O_{h, \epsilon}\big(X^{\frac{3}{4}+\epsilon}\big).
\end{align*}
\end{theorem}

As an immediate corollary, we obtain a two-square analogue
of~\eqref{eq:hooley_basic}.

\begin{corollary} \label{cor:cor_introduction}
  Fix $h \in \mathbb{Z}_{>0}$. Then for any $\epsilon > 0$, we have
  \begin{equation}
    \sum_{n^2 + m^2 + h \leq X}
    d(n^2 + m^2 + h)
    =
    c_h X \log X + c_h' X
    + O_{h,\epsilon}\big(X^{\frac{3}{4} + \epsilon}\big)
  \end{equation}
  for explicitly computable constants $c_h$ and $c_h'$.
\end{corollary}
We approach these sums by studying the meromorphic continuation of the
two-variable Dirichlet series
\begin{equation}\label{eq:basic_series}
  D_h(s, w)
  :=
  \sum_{n \geq 0}
  \frac{r_2(n) \sigma_{1 - 2w}(n + h)}{(n + h)^{s + \frac{1}{2} - w}}.
\end{equation}
This series arises as a shifted convolution involving automorphic
forms of weights $0$, $1$, and $-1$.
The idea of studying the behavior of $S(X)$ by constructing a shifted
convolution Dirichlet series builds on an observation from
Sarnak~\cite{sarnak1984additive}, who recognized that a Dirichlet series to
study $S^1(X)$ could be constructed via a Petersson inner product between a
theta function, an Eisenstein series, and a half-integral weight Poincar\'e
series.
Sarnak noted that this construction directly relates sums of natural,
arithmetic functions to the spectrum of the half-integral weight Laplacian,
but he did not pursue the resulting analysis.

The discrete spectrum of the Laplacian contributes to the dominating error term in Theorem~\ref{thm:main_introduction} and is the primary barrier to further improvement.
In our analysis in \S\ref{sec:Dh(s,w)-growth}, we relate the behavior of the discrete spectrum to growth estimates for certain spectral fourth moments.
Spectral fourth moments for $L$-functions of weight $0$ Maass forms are well-understood by
means of Iwaniec's spectral large sieve~\cite{Iwaniec92,Motohashi92}.
Similar results for weight $\pm 1$ (and more generally, for any odd weights --- which strongly
relate to each other through raising and lowering operators) are missing from the literature.

Specifically, we expect the following conjecture to hold.

\begin{conjecture}[Spectral Fourth Moment Conjecture for Weight $1$ Forms]%
\label{conj:intro-spectral-fourth-moment}
  Let $\{\mu_j^1\}_j$ denote an orthonormal basis of Maass forms of weight $1$ on
  $\Gamma_0(N)$, with Fourier expansions
  $
    \mu_j^1(z)
    =
    \sum_{m \neq 0}
    \rho_j^{1}(m)
    W_{\frac{m}{2\vert m \vert}, it_j}(4\pi \vert y)
    e(mx).
  $
  Fix $r \in \mathbb{R}$ and $\epsilon > 0$.
  As $T \to \infty$, we have
  \begin{equation*}
    \sum_{\vert t_j \vert \leq T}
    \frac{\vert \rho_j^{1}(1) \vert^2}{\cosh(\pi t_j)}
    \lvert L(\tfrac{1}{2}+ir,\mu_j^1) \rvert^4
    \ll_{N,r,\epsilon}
    T^{1+\epsilon}.
  \end{equation*}
\end{conjecture}

This conjecture is implied by the generalized Lindel\"{o}f hypothesis and is therefore well-supported. Indeed, Conjecture~\ref{conj:intro-spectral-fourth-moment} follows immediately from the generalized Lindel\"{o}f hypothesis and the Kuznetsov--Proskurin trace formula in weight $1$.

In Appendix~\ref{app:huang_kuan}, Huang and Kuan study these spectral fourth moments with methods
building on a weight $1$ trace formula
from~\cite{proskurin2003general, DFIsub} and a large sieve based on
integral estimates from~\cite{Humphries2016}.
This result is used in the proof of Theorem~\ref{thm:main_introduction} above.
Unfortunately, this analysis does not reach the conjectured bound.
We expect that it is possible to develop a suitable large sieve and prove the conjecture.
To that end, we also state (and prove) the following conditional improvement to
Theorem~\ref{thm:main_introduction}.

\begin{theorem} \label{thm:main_conditional_introduction}
  If Conjecture~\ref{conj:intro-spectral-fourth-moment} holds, then the error term in
  Theorem~\ref{thm:main_introduction} can be improved to
  \begin{equation}
    \mathrm{Err}_w^2(X) \ll_{h, w, \epsilon}
    \begin{cases}
      X^{\frac{1 + 2\Re w}{1 + 4 \Re w} + \epsilon}
      \qquad & \Re w \in [\frac{1}{2}, 1),
      \\
      X^{2 - 2 \Re w - \frac{2 - 2\Re w}{5 - 4 \Re w} + \epsilon}
      \qquad & \Re w \in (0, \frac{1}{2}].
    \end{cases}
  \end{equation}
\end{theorem}

\begin{corollary}
  If Conjecture~\ref{conj:intro-spectral-fourth-moment} holds, then for any fixed $h \in
  \mathbb{Z}_{>0}$ and $\epsilon > 0$, we have
  \begin{equation}
    \sum_{n^2 + m^2 + h \leq X}
    d(n^2 + m^2 + h)
    =
    c_h X \log X + c_h' X
    + O_{h,\epsilon}\big(X^{\frac{2}{3} + \epsilon}\big)
  \end{equation}
  for explicitly computable constants $c_h$ and $c_h'$.
\end{corollary}

\begin{remarks}\
  \begin{enumerate}
    \item We expect that any improvement in Theorem~\ref{thm:main_conditional_introduction} would be very hard to realize and require a much
    deeper understanding of the discrete spectrum.
    The obstructions to improvement appear already in Theorem~\ref{thm:Dh(s,w)-growth} (and there
    the primary obstruction is from Proposition~\ref{prop:ds-growth-k2}).

		\item We restrict to $\Re w \in (0,1)$ throughout, as this suffices for our main application, Corollary~\ref{cor:cor_introduction}. This restriction can be removed following Remark~\ref{rem:general-w}.
		We stress that the particular arithmetic results presented here are specific to $\Re w \in (0,1)$; more work would be required to correctly state the generalizations of Theorems~\ref{thm:main_introduction} and~\ref{thm:main_conditional_introduction} to $w$ outside the strip $\Re w \in (0,1)$.

    \item The authors note that it is possible to carry out the analogous work for the Dirichlet
    series obtained from~\eqref{eq:basic_series} after replacing $r_2$ by $r_k$.
    See for example Proposition~\ref{prop:vw_def}, which has obvious
    generalization.

    For $r_k = r_1$, this essentially yields the Dirichlet series proposed by Sarnak and
    would allow a different line of attack on the problem of Bykovskii.
    In unpublished work, the authors show that this method is unable to improve
    the error estimates~\eqref{eq:bykovskii_error} given by Bykovskii, with the
    discrete spectrum again representing the primary obstruction.
    We expect that any sharpened spectral result in either the analysis of
    $S^1(X)$ or $S(X)$ would lead to a corresponding improvement in the other.

    \item
    Bykovskii's studies positive definite quadratic forms with determinant $h$
    under the action of $\SL(2, \mathbb{Z})$ in order to study the $r_1$ case.
    This does not naturally generalize to the $r_2$ case; and though
    we both use spectral theory, the nature of our application is very
    different.
    One reason why our results have similar strength is because we both appeal
    to the Kuznetsov trace formula for averaged behavior of the discrete
    spectrum.
  \end{enumerate}
\end{remarks}

\section*{Outline of Paper}

In section~\ref{sec:Dh(s,w)-via-automorphic-forms}, we show that the series
$D_h(s,w)$ defined in~\eqref{eq:basic_series} can be expressed as a Petersson
inner product of various automorphic forms.
The functions involved have mild growth at cusps.
We regularize the functions and then study a spectral expansion for $D_h(s, w)$ in
section~\ref{sec:spectral-expansion}.

In section~\ref{sec:meromorphic_continuation}, we utilize this spectral
expansion to determine the meromorphic continuation of $D_h(s, w)$.
This culminates in Theorem~\ref{thm:D_h(s,w)-poles}, a full description of the poles
and residues of $D_h(s,w)$ in $\Re s >0$ for $0 < \Re w < 1$.
In section~\ref{sec:Dh(s,w)-growth}, we prove Theorem~\ref{thm:Dh(s,w)-growth}, which establishes
bounds for $D_h(s,w)$ in vertical strips.

In section~\ref{sec:sharp-cutoffs}, we apply standard Perron-type integral transforms to extract arithmetic information from $D_h(s,w)$. By leveraging the
analytic data collected in sections~\ref{sec:meromorphic_continuation} and~\ref{sec:Dh(s,w)-growth}, we prove our main arithmetic result, Theorem~\ref{thm:S(X)-growth}.

We give three appendices.
The first appendix is a computationally straightforward but tedious analysis of inner products with
Eisenstein series.
The second appendix gives an explicit decomposition of the weight $0$, level $1$ real analytic Eisenstein series as a sum of
the standard three weight $0$, level $4$ real analytic Eisenstein series.

The third appendix is far more substantial and was contributed by Huang and Kuan.
They prove a spectral fourth moment bound towards (but not attaining)
Conjecture~\ref{conj:intro-spectral-fourth-moment}.

\section*{Acknowledgements}

DLD was supported by the Simons Collaboration in Arithmetic
Geometry, Number Theory, and Computation via the Simons Foundation grant 546235.
Alexander Walker was supported by the Additional Funding Programme for Mathematical Sciences, delivered by EPSRC (EP/V521917/1) and the Heilbronn Institute for Mathematical Research.
C.I. Kuan is supported in part by NSFC (No. 11901585).
Huang and Kuan would also like to thank Sun Yat-Sen University for its warm environment.

We also thank Riad Masri, Yiannis Petridis, and Tom Hulse for their comments and suggestions.

\section{Expressing \texorpdfstring{$D_h(s, w)$}{Dh(s,w)} through automorphic forms}
\label{sec:Dh(s,w)-via-automorphic-forms}

Let $\theta(z) = \sum_{n \in \mathbb{Z}} e(n^2 z) = \sum_{n \geq 0} r_1(n) e(n z)$ denote the
classical theta function on the upper half-plane $\mathcal{H}$.
We write $e(z) = e^{2\pi i z}$ and use $r_k(n)$ to denote the number of
representations of $n$ as a sum of $k$ squares.
The theta function transforms following
\begin{equation*}
  \theta(\gamma z)
  =
  j(\gamma,z)\theta(z)
  =
  \epsilon_d^{-1} \Big(\frac{c}{d}\Big) (cz+d)^{\frac{1}{2}} \theta(z),
  \quad \gamma = \left(\begin{matrix} a & b \\ c & d \end{matrix}\right) \in \Gamma_0(4),
\end{equation*}
where $(\frac{c}{d})$ denotes the Kronecker symbol and $\epsilon_d = 1$ for $d \equiv 1 \bmod 4$ and $\epsilon_d = i$ for $d \equiv 3 \bmod 4$.
The classical theta function $\theta(z)$ is a modular form of weight
$\frac{1}{2}$ on $\Gamma_0(4)$ (see~\cite{Koblitz84} for supplemental background).

We also require Selberg's level $1$, weight $0$ real analytic Eisenstein series
\[
  E(z,w)
  :=
  \sum_{\gamma \in \Gamma_\infty \backslash \mathrm{SL}_2(\mathbb{Z})} \Im(\gamma z)^{w},
\]
described for example in~\cite[\S3]{Goldfeld06}. This Eisenstein series has the completion
$E^*(z,w) = \zeta^*(2w)E(z,w)$, where $\zeta^*(w)$ is the completed zeta function $\zeta^*(w) =
\pi^{-w/2} \Gamma(\frac{w}{2})\zeta(w)$. The completed Eisenstein series satisfies the functional
equation $E^*(z, w) = E^*(z, 1-w)$.

Lastly, for $h \in \mathbb{Z}_{>0}$, we require the
normalized weight $-1$ Poincar\'e series
$P^{-1}_h(z,s)$ attached to the $\infty$ cusp on $\Gamma_0(4)$, defined by
\[
  P^{-1}_h(z,s)
  :=
  \sum_{\gamma \in \Gamma_\infty \backslash \Gamma_0(4)}
  J(\gamma,z)^{2} \Im(\gamma z)^{s} e(h\gamma z),
\]
where $J(\gamma,z) =j(\gamma,z)/\vert j(\gamma,z) \vert$ is a normalized form of the typical
half-integral weight cocycle $j(\gamma,z)$ defined above.

We use these Poincar\'e series to construct the two-variable Dirichlet series
\begin{equation*}
  D_h(s,w)
  :=
  \sum_{m \geq 0} \frac{r_2(m)\sigma_{1-2w}(m+h)}{(m+h)^{s+\frac{1}{2}-w}}
\end{equation*}
by studying the Petersson inner product
\begin{align*}
  \bigl\langle
    \Im(z)^{\frac{1}{2}} \overline{\theta(z)}^{2} & E^*(z,w),
    P^{-1}_h(z,\overline{s})
  \bigr\rangle
  \\&=
  \int\limits_{\Gamma_0(4) \backslash \mathcal{H}}
  \Im(z)^{\frac{1}{2}} \overline{\theta(z)}^{2} E^*(z,w)
  \overline{P^{-1}_h(z,\overline{s})}
  d \mu(z),
\end{align*}
where $d \mu$ denotes the Haar measure normalized so that
$\mu\big(\Gamma_0(4) \backslash \mathcal{H}\big) = 1$.

\begin{proposition}\label{prop:pre-regularized-inner-product}
  For each $h \in \mathbb{Z}_{>0}$ and $\Re s > 1 + \lvert \Re w -\frac{1}{2} \rvert$,
  we have
  \[
    D_h(s,w)
    =
    \frac{%
      (4\pi)^{s-\frac{1}{2}} \Gamma(s+\frac{1}{2})
      \langle
      y^{\frac{1}{2}} \overline{\theta(z)^2} E^*(z,w),
      P^{-1}_h(z,\overline{s})
      \rangle
    }{%
      \Gamma(s-\frac{1}{2}+w)\Gamma(s+\frac{1}{2}-w)
    }.
  \]
\end{proposition}

\begin{proof}
As shown in~\cite[Theorem~3.1.8]{Goldfeld06}, the Eisenstein series $E(z,w)$
has Fourier expansion
\begin{equation}\label{eq:eis_wt0_fexp}
  E(z,w)
  =
  a_0(w,y)
  +
  \frac{2 \sqrt{y}}{\zeta^*(2w)}
  \sum_{n \neq 0} \frac{\sigma_{1-2w}(n)}{\vert n \vert^{\frac{1}{2}-w}}
  K_{w-\frac{1}{2}}(2\pi \vert n \vert y) e(nx),
\end{equation}
where the constant coefficient is $a_0(w,y) = y^w + y^{1-w}\zeta^*(2w-1)/\zeta^*(2w)$.

Expanding the Poincar\'e series and following the standard unfolding argument, we compute
that
\begin{align*}
  &\big\langle y^{\frac{1}{2}} \overline{\theta(z)^{2}} E^*(z,w),
  P^{-1}_h(z,\overline{s}) \big\rangle
  =
  \int_{\Gamma_0(4) \backslash \mathcal{H}}
  y^{\frac{1}{2}} \overline{\theta(z)^{2}}
  E^*(z,w) \overline{P^{-1}_h(z,\overline{s})} \; d\mu(z)
  \\
  &\quad =
  \int_0^\infty \int_0^1
  y^{s+\frac{1}{2}} \overline{\theta(z)^{2}}
  E^*(z,w) \overline{e(hz)} \, \frac{dx \, dy}{y^2}
  \\
  &\quad =
  \int_0^\infty \int_0^1 y^{s-\frac{1}{2}}
  \Big(
    \sum_{m_1 \geq 0}^\infty r_{2}(m_1) e^{-2\pi i m_1 x -2\pi m_1 y}
  \Big)
  e^{-2\pi h i x - 2\pi h y} \zeta^*(2w)
  \\
    &\qquad
    \times \Big(%
      a_0(w,y)
      +
      \frac{2 \sqrt{y}}{\zeta^*(2w)} \sum_{m_2 \neq 0} \!
      \frac{\sigma_{1-2w}(m_2)}
           {\vert m_2 \vert^{\frac{1}{2}-w}}
      K_{w-\frac{1}{2}}(2\pi \vert m_2 \vert y)
      e^{2\pi i m_2 x}
    \Big) \frac{dx \, dy}{y}.
\end{align*}
The $x$-integral extracts those terms with $m_2 - m_1 -h=0$ and leaves a single integral
in $y$ which can be simplified via~\cite[6.621(3)]{GradshteynRyzhik07}, giving
\begin{align*}
  \sum_{m \geq 0}
  &\frac{2 r_2(m)\sigma_{1-2w}(m+h)}
        {(m+h)^{s+\frac{1}{2}-w}(2\pi)^s}
  \int_0^\infty y^s e^{-y} K_{w-\frac{1}{2}}(y) \frac{dy}{y}
  \\
  &=
  \frac{2\sqrt{\pi} \, \Gamma(s-\frac{1}{2} + w) \Gamma(s+\frac{1}{2} - w)}
       {(4\pi)^s \Gamma(s + \frac{1}{2})}
  \sum_{m \geq 0} \frac{r_{2}(m)\sigma_{1-2w}(m+h)}{(m+h)^{s+\frac{1}{2}-w}}.
\end{align*}
Rearrangement completes the proof.
\end{proof}

\subsection{Regularization}

Proposition~\ref{prop:pre-regularized-inner-product} relates $D_h(s,w)$ to the
automorphic forms $\theta(z)$, $E(z,w)$, and $P^{-1}_h(z,s)$.
We will leverage automorphicity later by replacing $P^{-1}_h$ with its spectral
decomposition to obtain another description of $D_h(s, w)$.
Doing so directly would introduce convergence issues as
$y^{1/2} \overline{\theta(z)^2} E^*(z,w)$ is not in $L^2(\Gamma_0(4) \backslash \mathcal{H})$;
to avoid these issues, we use a related, regularized function $V_w(z)$ which is in
$L^2$.

To regularize $y^{1/2} \overline{\theta(z)^2} E^*(z,w)$ and obtain $V_w$, we subtract an appropriate linear combination of Eisenstein series of weight $-1$ attached to the three cusps $\infty$, $0$, and
$\frac{1}{2}$ of $\Gamma_0(4)$. We define the Eisenstein series of weight $-1$ at the infinite cusp of $\Gamma_0(4)$ as
\begin{equation}\label{eq:weighted-Eisenstein-definition}
  E_\infty^{-1}(z,w)
  =
  \sum_{\gamma \in \Gamma_\infty \backslash \Gamma_0(4)}
  J(\gamma,z)^2 \Im(\gamma z)^w.
\end{equation}
The Eisenstein series at the other cusps are given by
\[
  E_0^{-1}(z,w)
  =
  \Big(\frac{z}{\vert z \vert}\Big)^{-1}
  E_\infty^{-1}(\sigma_0 z,w),
  \quad
  E_{\frac{1}{2}}^{-1}(z,w)
  =
  \Big(\frac{2z+1}{\vert 2z+1 \vert}\Big)^{-1}
  E_\infty^{-1}(\sigma_{\frac{1}{2}} z,w),
\]
in which $\sigma_0 = (\begin{smallmatrix} 0 & -1 \\ 4 & 0 \end{smallmatrix})$ and
$\sigma_{1/2} = (\begin{smallmatrix} 1 & 0 \\ 2 & 1 \end{smallmatrix})$ denote the scaling
matrices at the cusps $0$ and $\frac{1}{2}$, respectively. We remark that these Eisenstein
series also appear (with slightly different normalization) in~\cite{GoldfeldHoffstein85}
and (with very different notation) in~\cite{shimura1975holomorphy}.

It will be useful to refer to behavior at other cusps in terms of the weight
$k$ slash operator as in~\cite{Shimura73}, except relative to the normalized
(half-integral weight) cocycle $J(\gamma, z)$, which we write here directly in
terms of the $\theta$ transformation law:
\begin{equation}
  \begin{split}\label{eq:slash-operator}
  f \big \vert^{k}_{[\gamma]}(z)
  & :=
  J(\gamma,z)^{-2k}  f(\gamma z)
  =
  \Big(
    \frac{\theta(\gamma z)}{\theta(z)}
    \cdot
    \frac{\lvert \theta(z) \rvert}{\lvert \theta(\gamma z ) \rvert}
  \Big)^{-2k}
  \cdot
  f(\gamma z)
  \\
  & =
  \Big(
    \frac{\theta(\gamma z)}{\theta(z)}
    \cdot
    \frac{\Im(\gamma z)^{\frac{1}{4}}}{\Im(z)^{\frac{1}{4}}}
  \Big)^{-2k}
  \cdot
  f(\gamma z).
\end{split}
\end{equation}

Each Eisenstein series has a Fourier--Whittaker expansion at each cusp of the form
\begin{equation}
  \begin{split}\label{eq:Eisenstein-Fourier}
    E_{\mathfrak{a}}^{-1}(z, w)
    \Big\vert^{-1}_{[\sigma_{\mathfrak{b}}]}
    &=
    \delta_{[\mathfrak{a}=\mathfrak{b}]} y^w +
    \rho_{\mathfrak{a},\mathfrak{b}}^{-1}(0,w)y^{1-w}
    \\
    &\quad +
    \sum_{m \neq 0} \rho_{\mathfrak{a},\mathfrak{b}}^{-1}(m,w)
    W_{\frac{m}{2\lvert m \rvert}, w-\frac{1}{2}}(4\pi \vert m \vert y)
    e(mx).
  \end{split}
\end{equation}

\begin{proposition}\label{prop:vw_def}
  Define $V_w(z)$ by
  \begin{align*}
    V_w(z)
    &:= y^{\frac{1}{2}} \overline{\theta(z)^2}E^*(z,w)
    \\
    &\qquad
      - \zeta^*(2w) E_\infty^{-1}\!(z,\tfrac{1}{2}+w)
        - \zeta^*(2-2w) E_\infty^{-1}\!(z,\tfrac{3}{2}-w)
    \\
    &\qquad
        - \zeta^*(2w)E_0^{-1}(z,\tfrac{1}{2}+w)
        - \zeta^*(2-2w) E_0^{-1}(z,\tfrac{3}{2} - w).
  \end{align*}
  Then for $w \neq \frac{1}{2}$ with
  $\Re w \in (0 , 1)$, we have
  $V_w(z) \in L^2(\Gamma_0(4) \backslash \mathcal{H}, -1)$,
  i.e.\ $V_w(z)$ is an $L^2$ automorphic form of weight $-1$ and level $4$.
\end{proposition}

\begin{proof}
It is clear that $V_w(z)$ transforms as an automorphic form of weight
$-1$ on $\Gamma_0(4)$.
It remains to verify that $V_w(z)$ has sufficient decay at each of the three cusps of $\Gamma_0(4)$.
We first address the growth at $\infty$. Recalling the expansion for $E(z, w)$
from~\eqref{eq:eis_wt0_fexp}, the estimate
\[
  y^{\frac{1}{2}} \overline{\theta(z)^2}E^*(z,w)
  =
  \zeta^*(2w) y^{\frac{1}{2}+w}
  +
  \zeta^*(2-2w)y^{\frac{3}{2} - w}
  +
  O_w\big(y^{\frac{1}{2}} e^{-2\pi y}\big)
\]
can be seen directly and gives the leading order behavior of the non-regularized term
as $y=\Im(z) \to \infty$ for $w \neq \frac{1}{2}$. We've used the classical
asymptotic~\cite[8.451]{GradshteynRyzhik07} $K_\nu(y) = O(e^{-y}/\sqrt{y})$ to collect the
error terms.

The four subtracted weight $-1$ Eisenstein series in the definition of
$V_w(z)$ behave as
\[
  \zeta^*(2w)y^{\frac{1}{2}+w}
  +
  \zeta^*(2-2w)y^{\frac{3}{2}- w}
  +
  O_w\big(y^{\frac{1}{2}-w}
  +
  y^{-\frac{1}{2}+w}
  +
  y^{\frac{1}{2}} e^{-2\pi y}\big),
\]
from~\eqref{eq:Eisenstein-Fourier} and the bound~\cite[9.227]{GradshteynRyzhik07} $W_{\alpha,
\nu}(y) = O_{\alpha, \nu}(y^\alpha e^{-y/2})$ as $y \to \infty$.

It follows that $V_w(z) = O_w(y^{\frac{1}{2}-w} + y^{w-\frac{1}{2}})$
for $w \neq \frac{1}{2}$ as $y\to \infty$.
The assumption $\Re w \in (0, 1)$ then guarantees
$V_w(z) = o(\sqrt{y})$ as $y\to \infty$, which is enough for convergence near the cusp at $\infty$.

Understanding behavior at the cusp at $0$ is analogous, as $E^*(z,w)$ is invariant under $\SL(2, \mathbb{Z})$ and the involution $\theta(-1/4z) =
(-2iz)^{1/2} \theta(z)$ can be used to show that $\theta
\big|_{[\sigma_0]}^{1/2}(z) = \theta(z)$.

Mild growth at the cusp at $\frac{1}{2}$ can be seen from the fact that
$\theta(z)$ decays exponentially near $z=\frac{1}{2}$, while the subtracted Eisenstein series satisfy
\begin{align*}
  E^{-1}_\mathfrak{a}(z, \tfrac{1}{2}+w)
  \big\vert_{[\sigma_{\frac{1}{2}}]}^{-1}
  &=
  O_w\big(y^{\frac{1}{2} - w}\big),
  \\
  E^{-1}_\mathfrak{a}(z, \tfrac{3}{2}-w)
  \big\vert_{[\sigma_{\frac{1}{2}}]}^{-1}
  &=
  O_w\big(y^{w-\frac{1}{2}}\big)
\end{align*}
for $\mathfrak{a} = \infty$ and $0$.
\end{proof}

We relate the new inner product $\langle V_w, P^{-1}_h \rangle$ to the original
inner product $\langle y^{1/2} \overline{\theta(z)^2} E^*(\cdot, w),P^{-1}_h\rangle$ by explicitly
extracting the contribution of the additional Eisenstein series.
For $h$ a positive integer, we unfold the Poincar\'e series to evaluate
\begin{align*}
  \big\langle E_\mathfrak{a}^{-1}(z,w),P^{-1}_h(z,\overline{s}) \big\rangle
  &=
  \rho_\mathfrak{a}^{-1}(h,w)
  \int_0^\infty y^{s-1} e^{-2\pi h y}
  W_{-\frac{1}{2},w-\frac{1}{2}}(4\pi h y) \frac{dy}{y}
  \\
  &=
  \frac{\Gamma(s-1+w)\Gamma(s-w)}{(4\pi h)^{s-1}\Gamma(s+\frac{1}{2})}
  \rho_\mathfrak{a}^{-1}(h,w),
\end{align*}
in which $\rho_{\mathfrak{a}}^{-1}(h,w) = \rho_{\mathfrak{a},\infty}^{-1}(h,w)$
is the $h$-th coefficient of $E^{-1}_\mathfrak{a}$, and the second equality follows
from the integral evaluation~\cite[7.621(3)]{GradshteynRyzhik07}.
By applying this identity to Proposition~\ref{prop:pre-regularized-inner-product}, we
prove the following.

\begin{proposition}\label{prop:regularized-series}
  For $\Re(s) > 1 + \lvert \Re w - \frac{1}{2} \rvert$, we have
  \begin{align*}
    D_h(s,w)
    &=
    \frac{(4\pi)^{s-\frac{1}{2}}
          \bigl\langle V_w, P^{-1}_h(\cdot, \overline{s})\bigr\rangle
          \Gamma(s+\frac{1}{2})}
         {\Gamma(s-\frac{1}{2}+w)\Gamma(s+\frac{1}{2}-w)}
    \\
    &\qquad + \frac{(4\pi)^{\frac{1}{2}}\zeta^*(2w)\Gamma(s-w-\frac{1}{2})}
                   {h^{s-1}\Gamma(s+\frac{1}{2}-w)}
    \varphi_h(\tfrac{1}{2}+w)
    \\
    &\qquad + \frac{(4\pi)^{\frac{1}{2}}\zeta^*(2-2w) \Gamma(s+w-\frac{3}{2})}
                   {h^{s-1}\Gamma(s-\frac{1}{2}+w)}
    \varphi_h(\tfrac{3}{2}-w),
  \end{align*}
  in which $\varphi_h(w) = \rho_\infty^{-1}(h,w)
  + \rho_0^{-1}(h,w)$.
\end{proposition}

The Fourier coefficients $\rho_{\mathfrak{a}}^{-1}(h,w)$ in
Proposition~\ref{prop:regularized-series} can be computed explicitly, but we
require only a general description here.
  (See \S2.2.1 of~\cite{LowryDuda17} for one such explicit derivation; with a
  different choice of weight normalization, these coefficients are also
  described in Propositions~1.2, 1.4, and~1.5 of~\cite{GoldfeldHoffstein85}).
The general shape of the arithmetic portion of the $h$-th coefficient is a generalized
divisor function divided by a Dirichlet $L$-function.

\begin{lemma}\label{lem:Eisenstein-series-coefficients-formulas-wt1}
  For $h > 0$, the $h$-th Fourier coefficient of $E_{\mathfrak{a}}^{-1}(z,w)$ equals
  \begin{align}\label{eq:rho_a-formula-wt1}
    \rho_{\mathfrak{a}}^{-1}(h,w)
    =
    d_{\mathfrak{a}}' d_{\mathfrak{a}}^w
    \frac{\pi^w h^{w-1}}{\Gamma(w - \frac{1}{2})}
    \cdot \frac{Q_\mathfrak{a}^{-1}(h, w)}{L^{(2)}(2w, (\frac{-1}{\cdot}))}
  \end{align}
  in which $Q_\mathfrak{a}^{-1}(h, w)$ is a finite Dirichlet polynomial depending on the cusp $\mathfrak{a}$
  and the divisors of $h$; $d_\mathfrak{a}, d_\mathfrak{a}'$ are constants depending on
  the cusp; and the parenthetical superscript in $L^{(2)}$ indicates that the $2$-part of
  the Euler product is omitted.
\end{lemma}

\section{Spectral Expansion}\label{sec:spectral-expansion}

To understand the meromorphic continuation of $D_h(s, w)$, we study a spectral expansion of
$\langle V_w, P^{-1}_h \rangle$ via the spectral expansion of $P^{-1}_h$.

As $P^{-1}_h$ has weight $-1$, this expansion takes the form
\begin{align}\label{eq:poincare-spectral-expansion}
  P^{-1}_h(z,s)
  &=
  \sum_j \langle  P^{-1}_h(\cdot , s), \mu^{-1}_j \rangle
  \mu^{-1}_j(z)
  \\
  &\quad +
  \frac{1}{4\pi i} \sum_{\mathfrak{a}} \int_{(0)}
  \!\big\langle P^{-1}_h(\cdot, s), E^{-1}_\mathfrak{a}(\cdot,\tfrac{1}{2}+u)\big\rangle
  E^{-1}_\mathfrak{a}(z,\tfrac{1}{2}+u)\,du,
  \notag
\end{align}
in which $\{\mu^{-1}_j\}$ denotes an orthonormal eigenbasis of weight $-1$ Maass forms on $\Gamma_0(4)$.
We will refer to the two parts at right in~\eqref{eq:poincare-spectral-expansion} as the discrete
spectrum and the continuous spectrum, respectively.
As $P^{-1}_h$ has no constant term, it is orthogonal to the constant form $\mu^{-1}_0$,
which we henceforth ignore.
(For background details, see~\cite[\S15.3.7]{cohen2017modular}).

Each Maass form $\mu^{-1}_j$ is an eigenfunction of the weight $-1$ hyperbolic Laplacian
$\Delta = -y^2(\frac{\partial^2}{\partial y^2} + \frac{\partial^2}{\partial y^2}) -
iy \frac{\partial}{\partial x}$ and has an associated eigenvalue
$\frac{1}{4}+t_j^2$, spectral type $\frac{1}{2} \pm it_j$, and Fourier--Whittaker expansion
\[
  \mu^{-1}_j(z)
  =
  \sum_{m \neq 0} \rho_j^{-1}(m)
  W_{\frac{-m}{2\lvert m \rvert},it_j}(4\pi \lvert m \rvert y)e(mx).
\]

The inner products in~\eqref{eq:poincare-spectral-expansion} can be written
explicitly by unfolding the Poincar\'e series and
applying~\cite[7.621(3)]{GradshteynRyzhik07} several times. This routine but
tedious set of computations produces
\begin{align*}
  \langle P^{-1}_h(\cdot,s),\mu^{-1}_j\rangle
  &=
  \frac{\Gamma(s-\frac{1}{2}+it_j)\Gamma(s-\frac{1}{2}-it_j)}
       {(4\pi h)^{s-1}\Gamma(s+\frac{1}{2})}
  \, \overline{\rho_{j}^{-1}(h)},
  \\
  \langle P^{-1}_h(\cdot,s), E^{-1}_\mathfrak{a}(\cdot,\overline{w})\rangle
  &=
  \frac{\Gamma(s-1+w)\Gamma(s-w)}{(4\pi h)^{s-1}\Gamma(s+\frac{1}{2})}
  \, \overline{\rho_{\mathfrak{a}}^{-1}(h,\overline{w})}.
\end{align*}
Substituting~\eqref{eq:poincare-spectral-expansion} into
Proposition~\ref{prop:regularized-series} gives the spectral expansion.

\begin{proposition}\label{prop:D_h-spectral}
 For $\Re s > \frac{1}{2} + \lvert \Re w - \frac{1}{2} \rvert$ and
 $\Re w \in (0, 1)$ with $w \neq \frac{1}{2}$, we have
  \begin{equation}\label{eq:Dhsw_base_spectral}
    D_h(s, w)
    :=
    \sum_{m \geq 0} \frac{r_2(m) \sigma_{1-2w}(m+h)}{(m+h)^{s+\frac{1}{2}-w}}
    =
    \Sigma_{\mathrm{cont}}
    +
    \Sigma_{\mathrm{disc}}
    +
    \Sigma_{\mathrm{reg}}
  \end{equation}
  where the continuous part $\Sigma_{\mathrm{cont}}$ is given by
  \begin{align*}
    \Sigma_{\mathrm{cont}}(s,w)
    :=
    \frac{(4\pi)^{-\frac{1}{2}}} {h^{s-1} i}
    \sum_{\mathfrak{a}} \! \int_{(0)}
    & \frac{\Gamma(s - \frac{1}{2} + u)\Gamma(s - \frac{1}{2} - u)}
           {\Gamma(s - \frac{1}{2} + w)\Gamma(s + \frac{1}{2} - w)}
    \\
    & \quad \times
    \rho_{\mathfrak{a}}^{-1}(h, \tfrac{1}{2} + u)
    \langle V_w, E_{\mathfrak{a}}^{-1}(\cdot, \tfrac{1}{2} + \overline{u})\rangle du,
  \end{align*}
  the discrete part $\Sigma_{\mathrm{disc}}$ is given by
  \begin{equation*}
    \Sigma_{\mathrm{disc}}(s,w)
    :=
    \frac{(4\pi)^{\frac{1}{2}}}{h^{s-1}}
    \sum_j \frac{\Gamma(s-\tfrac{1}{2}+it_j)\Gamma(s-\frac{1}{2}-it_j)}
                {\Gamma(s-\frac{1}{2}+w)\Gamma(s+\frac{1}{2}-w)}
    \rho_j^{-1}(h) \langle V_w, \mu^{-1}_j \rangle,
  \end{equation*}
  and the terms $\Sigma_{\mathrm{reg}}$ coming from the ``non-spectral'' terms (i.e.\ the
  subtracted Eisenstein series from the regularization $V_w$) are given by
  \begin{equation*}
  \begin{split}
    \Sigma_{\mathrm{reg}}(s,w)
    &:=
    \frac{(4\pi)^{\frac{1}{2}} \zeta^*(2w)\Gamma(s - w - \frac{1}{2})}
         {h^{s-1}\Gamma(s + \frac{1}{2} - w)}
    \varphi_h(\tfrac{1}{2} + w)
    \\
    &\quad + \frac{(4\pi)^{\frac{1}{2}}\zeta^*(2-2w) \Gamma(s + w - \frac{3}{2})}
                   {h^{s-1}\Gamma(s-\frac{1}{2}+w)}
    \varphi_h(\tfrac{3}{2} - w).
  \end{split}
  \end{equation*}
\end{proposition}

In this proposition, the restriction to $\Re w \in (0, 1)$ and $w \neq \frac{1}{2}$ comes
from Proposition~\ref{prop:vw_def}: when splitting up the spectral expansion, we have
implicitly used that the expansion converges, and in particular that $\langle V_w,
E_\mathfrak{a}^{-1} \rangle$ converges.
This is only apparent when $V_w$ is in $L^2$.
In the next section, we show that $D_h(s, w)$ meromorphically extends to $w = \frac{1}{2}$
and to all $s \in \mathbb{C}$. We also justify the convergence of the discrete and
continuous spectra.

\section{Meromorphic Continuation}\label{sec:meromorphic_continuation}

We obtain the meromorphic continuation of $D_h(s, w)$ with respect to $s$ by continuing the discrete and
continuous spectra in~\eqref{eq:Dhsw_base_spectral} separately.
Here and later, $\chi = (\frac{-4}{\cdot})$ is the primitive Dirichlet character of conductor $4$.
The rest of this section is devoted to proving the following theorem.

\begin{theorem}\label{thm:D_h(s,w)-poles}
  Fix $\Re w \in (0, 1)$.
  The function $D_h(s,w)$, originally defined as a series convergent
  for $\Re s > 1 + \lvert \Re w - \frac{1}{2} \rvert$, has meromorphic continuation
  to all $s \in \mathbb{C}$.
  In the half-plane $\Re s > 0$, $D_h(s,w)$ has potential poles
  at $s = \frac{1}{2} + w$ and $s = \frac{3}{2} - w$ (from
  $\Sigma_{\mathrm{reg}}$),
  at $s = \frac{1}{2} \pm it_j$ for each type $t_j$ of a
  Maass form $\mu_j$ appearing in the spectral expansion (from $\Sigma_{\mathrm{disc}}$),
  and at $s = \rho_\chi/2$ for each zero $\rho_\chi$ of $L(s, \chi)$
  (from the first residual terms in $\Sigma_{\mathrm{cont}}$, as defined in~\eqref{eq:first_residual}).

  The polar behavior at $s=\frac{1}{2} + w$ and $s=\frac{3}{2}-w$ depends on whether
  $w=\frac{1}{2}$ or $w \neq \frac{1}{2}$:
  \begin{enumerate}
    \item[1.]
    For $w \neq \frac{1}{2}$, the poles at $s=\frac{1}{2} + w$ and
    $s=\frac{3}{2} - w$ are simple and
    \begin{align*}
      \Res_{s = \frac{1}{2} + w} D_h(s,w)
      &=
      \frac{(4\pi)^{\frac{1}{2}}\zeta^*(2w)}
           {h^{w - \frac{1}{2}}}
      \varphi_h(\tfrac{1}{2} + w),
      \\
      \Res_{s=\frac{3}{2} - w} D_h(s,w)
      &=
      \frac{(4\pi)^{\frac{1}{2}}\zeta^*(2-2w)}
           {h^{\frac{1}{2}-w}}
      \varphi_h(\tfrac{3}{2} - w),
    \end{align*}
    in which $\varphi_h(u) = \rho_\infty^{-1}(h,u)
    + \rho_0^{-1}(h,u)$.

    \vspace*{0.5em}

    \item[2.]
    For $w=\frac{1}{2}$, $D_h(s,\frac{1}{2})$ has a double pole at
    $s=1$ with principal part
    \begin{equation*}
      \frac{(4 \pi)^{\frac{1}{2}} \varphi_h(1)}
           {(s-1)^2}
      +
      \frac{(4 \pi)^{\frac{1}{2}}
               \big(%
                 (\gamma -\log(4\pi h))
                 \varphi_h(1)
                 +
                 (\varphi_h)'(1)
               \big)
           }{%
              (s-1)
           }.
    \end{equation*}
  \end{enumerate}
  The function $D_h(s,w)$ is otherwise holomorphic in $\Re s > 0$.
\end{theorem}

\subsection{Continuation of the continuous spectrum}\label{ssec:cont_continuation}

We first consider the contribution from $\Sigma_{\mathrm{cont}}$, which we recall equals
\begin{equation}
  \begin{split} \label{eq:continuous-spectrum}
    \Sigma_{\mathrm{cont}}(s,w)
    =
    &\frac{(4\pi)^{-\frac{1}{2}}} {h^{s-1} i}
    \sum_{\mathfrak{a}} \! \int_{(0)}
    \frac{\Gamma(s-\frac{1}{2}+u)\Gamma(s-\frac{1}{2}-u)}
         {\Gamma(s-\frac{1}{2}+w)\Gamma(s+\frac{1}{2}-w)}
    \\
    &\qquad\qquad\qquad \times
    \rho_{\mathfrak{a}}^{-1}(h, \tfrac{1}{2} + u)
    \langle
      V_w,
      E_{\mathfrak{a}}^{-1}(\cdot,\tfrac{1}{2} + \overline{u})
    \rangle du.
  \end{split}
\end{equation}
We must understand the inner products
$\langle V_w, E_\mathfrak{a}^{-1}(\cdot,\overline{u})\rangle$
in order to study $\Sigma_{\mathrm{cont}}$.
To this end, we supply the following proposition.

\begin{proposition}\label{prop:V-E_a-inner-product}
  Let $\chi = ( \frac{-4}{\cdot})$ be the primitive Dirichlet character of modulus $4$ and
  let $\zeta^*(s)$ and $L^*(s, \chi)$ denote the completed zeta and Dirichlet
  $L$-functions.
  We have
  \begin{align}
    \big\langle V_w,E_{\infty}^{-1}(\cdot,\overline{u}) \big\rangle
    &=
      \frac{2}{\sqrt{\pi}}
      \cdot
      \zeta^*(u-w+\tfrac{1}{2})\zeta^*(u+w-\tfrac{1}{2})
      \\
      &\qquad \times \frac{L^*(u-w+\frac{1}{2},\chi) L^*(u+w-\frac{1}{2},\chi)}{L^*(2u,\chi)},
    \\
      \big\langle V_w,E_{0}^{-1}(\cdot,\overline{u}) \big\rangle
      &=
      \frac{e(\frac{-1}{4})}{2^{2u-1}\sqrt{\pi}}
      \cdot
      \zeta^*(u-w+\tfrac{1}{2})\zeta^*(u+w-\tfrac{1}{2})
      \\
      &\qquad \times \frac{L^*(u-w+\frac{1}{2},\chi) L^*(u+w-\frac{1}{2},\chi)}{L^*(2u,\chi)},
    \\
      \big\langle V_w,E_{\frac{1}{2}}^{-1}(\cdot,\overline{u}) \big\rangle
      &=
      0.
  \end{align}
\end{proposition}

The proof of this proposition is purely computational, but a bit tedious.
For simplicity of presentation, we defer the proof to
Appendix~\ref{sec:V-E_a-inner-products}.

By Stirling's approximation and classical growth estimates, the gamma ratios and inner
products from Proposition~\ref{prop:V-E_a-inner-product} give large exponential decay in
$\lvert \Im u \rvert$ that overcomes potential growth from the Eisenstein coefficients
$\rho_{\mathfrak{a}}^{-1}$ (cf.\ Lemma~\ref{lem:Eisenstein-series-coefficients-formulas-wt1}).
It follows that~\eqref{eq:continuous-spectrum} is analytic for $\Re s > \frac{1}{2}$, but
there are potential poles on the line $\Re s = \frac{1}{2}$ from the gamma functions in
the numerator of the integrand.

We now briefly describe how to obtain the meromorphic continuation of $\Sigma_{\mathrm{cont}}$.
This is complicated by the entangled nature of the poles in $s$ with poles in the
integration variable $u$ in the integrand of~\eqref{eq:continuous-spectrum}.
To overcome this difficulty, one uses an iterative process of contour shifting and local adjustment.
This process is well understood and well-documented in the literature.
We adopt notation from~\cite{HoffsteinHulse13} and~\cite{HulseKuanLowryDudaWalker17}.
(See also~\cite{HKLDWSphere}, which presents this argument in a more similar case, though in
less detail).

For $\epsilon$ sufficiently small, suppose $\Re s \in (\frac{1}{2}, \frac{1}{2} + \epsilon)$.
We shift the $u$-contour to the left, along a contour $C_s$.
Recall from Lemma~\ref{lem:Eisenstein-series-coefficients-formulas-wt1} that
$E_{\mathfrak{a}}^{-1}(z, w)$ and its $h$-th Fourier coefficient $\rho_{\mathfrak{a}}^{-1}(h, w)$
have potential poles at zeros of $L(2w, \chi)$, hence
$\rho_{\mathfrak{a}}^{-1}(h, \frac{1}{2} + u)$ has potential poles at zeros of $L(2u+1,\chi)$.
We therefore choose a contour $C_s$ that bends to remain in the zero-free region of $L(2u+1,
\chi)$, but which passes the pole at $u = \frac{1}{2} - s$.
This is always possible if $\epsilon$ is chosen to be sufficiently small (dependent on $\Im s$).

The shifted integral is clearly meromorphic for $s$ to the right of $C_s$, giving a small
continuation in $s$.
For $s$ to the right of $C_s$ with $\Re s < \frac{1}{2}$, shifting the line of $u$-integration
back to $\Re u = 0$ extracts a second residual term, this time at $u = s - \frac{1}{2}$.
This integral has clear continuation to $\Re s > \frac{1}{2} - 1$. In total, these local continuations
give meromorphic continuation to $\Re s > -\frac{1}{2}$ and introduce two residual terms that appear
only for $\Re s < \frac{1}{2}$.
The two residual terms total
\begin{equation}
\begin{split}\label{eq:first_residual}
  \frac{(4\pi)^{\frac{1}{4}}\Gamma(2s-\frac{1}{2}) h^{1-s}}
       {\Gamma(s-\frac{1}{2}+w)\Gamma(s+\frac{1}{2}-w)}
  &\sum_{\mathfrak{a}}
  \Big(
    \rho_{\mathfrak{a}}^{-1}(h, 1-s)
    \langle V_w, E_{\mathfrak{a}}^{-1}(\cdot, \overline{s})\rangle
    \\
    &\qquad + \rho_{\mathfrak{a}}^{-1}(h,\tfrac{1}{4}+s)
    \langle V_w, E_{\mathfrak{a}}^{-1}(\cdot, 1 - \overline{s})\rangle
  \Big).
\end{split}
\end{equation}
We note again by Proposition~\ref{prop:V-E_a-inner-product} that the inner products
in~\eqref{eq:first_residual} have potential poles coming from $L^*(2s, \chi)$ in the denominator.
Thus $\Sigma_{\mathrm{cont}}$ has potential poles at $s = \rho_\chi/2$ for each
zero $\rho_\chi$ of the Dirichlet $L$-function.

Iterating this process, we meromorphically continue $\Sigma_{\mathrm{cont}}$ to all $s \in
\mathbb{C}$ (by adding an increasing number of residual terms).
For our primary application, we do not need this continuation to be made explicit.

\begin{remark}
  We note that to analyze the meromorphic behavior \emph{on the line } $\Re s = \frac{1}{2}$,
  one studies the local meromorphic continuation for $s$ to the right of $C_s$ determined
  from the shift of the contour integral to $C_s$.
  This involves the shifted contour integral and a single residual term.
\end{remark}

\subsection{Continuation of the discrete spectrum}\label{ssec:disc_continuation}

The contribution of the discrete spectrum in $D_h(s,w)$ takes the form
\begin{equation}\label{eq:discrete-spectrum}
  \Sigma_{\mathrm{disc}}(s,w)
  =
  \frac{(4\pi)^{\frac{1}{2}}}{h^{s-1}}
  \sum_j \frac{\Gamma(s-\tfrac{1}{2}+it_j)\Gamma(s-\frac{1}{2}-it_j)}
              {\Gamma(s-\frac{1}{2}+w)\Gamma(s+\frac{1}{2}-w)}
  \rho_j^{-1}(h) \langle V_w, \mu^{-1}_j \rangle.
\end{equation}

We will show that for each $w$ with $0 < \Re w < 1$, the Maass form components $\rho_j^{-1}(h)
\langle V_w, \mu^{-1}_j \rangle$ are analytic in $w$.
By estimating the growth of these terms sufficiently well, we also prove that the $j$-sum in
$\Sigma_{\mathrm{disc}}$ decays exponentially in $t_j$ for $\lvert t_j \rvert > \lvert s \rvert$.

We break our analysis of $\Sigma_{\mathrm{disc}}$ into steps.
\begin{enumerate}
  \item[1.] We decompose the weight $0$, level $1$ Eisenstein series as a
  sum of level $4$ Eisenstein series. With this, we can explicitly
  write $\langle V_w, \mu^{-1}_j \rangle$ as Dirichlet series.

  \item[2.] We then recognize these Dirichlet series in terms of standard
  $L$-functions in order to study their meromorphic behavior.

  \item[3.] To study $\rho_j^{-1}(h)$, we use on-average estimates derived from an application of the Kuznetsov trace formula to produce a crude bound. These estimates will be refined later.
\end{enumerate}

\subsubsection{Decomposing Eisenstein series}

To understand the Petersson inner product $\langle V_w,\mu^{-1}_j \rangle$, recall that the Eisenstein
series $E_{\mathfrak{a}}^{-1}$ used to regularize $V_w$ are orthogonal to the Maass forms
$\mu^{-1}_j$ and thus
\begin{equation*}
  \big\langle V_w, \mu^{-1}_j \big\rangle
  =
  \big\langle y^{\frac{1}{2}} \overline{\theta(z)^{2}}E^*(z,w),\mu^{-1}_j(z) \big\rangle.
\end{equation*}
This inner product may be understood using the Rankin--Selberg method, either by unfolding $E^*(z,w)$ or by writing $E_{\infty}^{-1}(z, \frac{1}{2}) =
  y^{1/2} \overline{\theta^2(\overline{z})}$ and unfolding $E_\infty^{-1}(z,\frac{1}{2})$.

\begin{remark}
  Unfolding the Eisenstein series $E_\infty^{-1}(z,\frac{1}{2})$ relates $\langle V_w,\mu^{-1}_j \rangle$ to products of
  $L$-functions and ${}_3F_2$-hypergeometric functions on the edge of convergence. The authors did
  not pursue this line of analysis.
\end{remark}

Meanwhile, unfolding with $E^*(z,w)$ is complicated by the fact that the inner product $\langle V_w, \mu_j^{-1} \rangle$ is defined over $\Gamma_0(4) \backslash \mathcal{H}$, whereas $E^*(z, w)$ is a level $1$ Eisenstein series. To resolve this mismatch, we write the level $1$ Eisenstein series $E(z, w)$ as a sum of Eisenstein series of weight $0$ and level $4$.
Specializing~\eqref{eq:weighted-Eisenstein-definition} to weight $0$, we define
\begin{equation*}
  E_\infty(z, w)
  =
  \sum_{\gamma \in \Gamma_\infty \backslash \Gamma_0(4)}
  \Im(\gamma z)^s,
\end{equation*}
and define $E_\mathfrak{a}(z, w) = E_\infty(\sigma_\mathfrak{a} z, w)$. We also define the completed
$\Gamma_0(4)$ Eisenstein series $E_\mathfrak{a}^*(z, w) = \zeta^*(2w) E_{\mathfrak{a}}(z, w)$,
agreeing with the completed $\SL(2, \mathbb{Z})$ Eisenstein series $E^*(z, w)$.

\begin{proposition}\label{prop:Elevel1_is_Elevel4_sum}
  We have $E(z,w) = E_\infty(z,w) + 4^w E_0(z,w) + E_{1/2}(z,w)$.
\end{proposition}

This proof is a straightforward but un-illuminating computation.
For ease of presentation, we defer the proof to Appendix~\ref{sec:Eisenstein-decomposition}.

To understand $\langle V_w, \mu^{-1}_j \rangle$, it suffices to understand the inner product
$\langle y^{1/2} \overline{\theta(z)^{2}} E_\mathfrak{a}(z,w),\mu^{-1}_j(z)\rangle$ for each cusp
$\mathfrak{a}$ of $\Gamma_0(4)$.
To do this, we treat each inner product via the Rankin--Selberg method, obtaining
Rankin--Selberg convolutions of Maass forms against $\theta(z)^2$.

To describe these convolutions, it will be useful to use expansions of the Maass forms at other cusps.
We continue to use the weight $k$ slash operator with respect to the normalized
$J$ cocycle, as in~\eqref{eq:slash-operator}.
Note also that the Laplacian commutes with the slash operator, hence the space of Maass
forms is preserved.
The forms $\mu_j^{-1} \vert^{-1}_{[\sigma_{\mathfrak{a}}^{-1}]}$
each have Fourier--Whittaker expansions, and we
will express our convolutions in terms of the coefficients of these expansions.

The terms coming from $E_\infty$ and $E_0$ are straightforward to understand.
We directly compute that
\begin{align}
  \langle
    \Im (z)^{\frac{1}{2}} \overline{\theta(z)^{2}} &
    E_\mathfrak{a} (z, w),
    \mu_j^{-1} (z)
  \rangle
  =
  \langle
    E_\infty (\sigma_\mathfrak{a} z, w),
    \Im (z)^{\frac{1}{2}} \theta(z)^{2}
    \mu_j^{-1} (z)
  \rangle \notag{}
  \\
  &=
  \Big\langle
    E_\infty (z, w),
    \frac{%
      \Im (\sigma_{\mathfrak{a}}^{-1} z)^{\frac{1}{2}}
      \theta(\sigma_{\mathfrak{a}}^{-1} z)
    }{J(\sigma_{\mathfrak{a}}, z)^{2}}
    \frac{%
      \mu_j^{-1} (\sigma_{\mathfrak{a}}^{-1} z)
    }{J(\sigma_{\mathfrak{a}}, z)^{-2}}
  \Big\rangle \notag{}
  \\
  &=
  \Big\langle
    E_\infty (z, w),
    \big( \Im(z)^{\frac{1}{2}} \theta(z)^{2} \big)
    \big\vert^{1}_{[\sigma_{\mathfrak{a}}^{-1}]}
    \big( \mu_j^{-1}(z) \big)
    \big\vert^{-1}_{[\sigma_{\mathfrak{a}}^{-1}]}
  \Big\rangle.\label{eq:Eis_cusp_is_exp_at_cusp}
\end{align}
For both $\mathfrak{a} = 0$ and $\mathfrak{a} = \infty$, we find that
$(\Im(z)^{\frac{1}{2}} \theta(z)^{2} ) \vert^1_{[\sigma_{\mathfrak{a}}^{-1}]}
=
\Im(z)^{\frac{1}{2}} \theta(z)^{2}$;
this is immediate for $\mathfrak{a} = \infty$, and for $\mathfrak{a} = 0$ this follows from the
involution $\theta(-1/4z) = (-2iz)^{1/2} \theta(z)$.

The stabilizer group $\Gamma_0 \subset \Gamma_0(4)$ is generated by
\begin{equation*}
 \gamma_0 :=
 \begin{pmatrix}
    0 & -1 \\ 4 & 0
  \end{pmatrix}
  \begin{pmatrix}
    1 & 1 \\ 0 & 1
  \end{pmatrix}
  \begin{pmatrix}
    0 & -1 \\ 4 & 0
  \end{pmatrix}^{-1}
  =
  \begin{pmatrix}
    1 & 0 \\ -4 & 1
  \end{pmatrix}.
\end{equation*}
Since $J(\gamma_0, z) = 1$, we see that
$\mu_{j, 0}^{-1} := \mu_j^{-1} \vert^{-1}_{[\sigma_0^{-1}]}$
is invariant under $z \mapsto z + 1$ and has a Fourier--Whittaker expansion of the form
\begin{equation}
  \mu_{j, 0}^{-1} := \mu_j^{-1} \big
  \vert^{-1}_{[\sigma_0^{-1}]}
  =
  \sum_{m \neq 0} \rho_{j, 0}^{-1}(m)
  W_{\frac{-m}{2 \lvert m \rvert}, it_j}(4 \pi \lvert m \rvert y) e(mx).
\end{equation}
We write $\mu_{j, \infty}^{-1} := \mu_j^{-1}$, with coefficients $\rho_{j, \infty}^{-1}$, for the
trivial analogous statement at the $\infty$ cusp.

Thus, for $\mathfrak{a} \in \{0,\infty\}$, a standard unfolding computation shows that
\begin{align*}
    \langle
      y^{\frac{1}{2}} \overline{\theta(z)^{2}} E_{\mathfrak{a}}(z, w)& ,
      \mu_j^{-1}(z)
    \rangle
    =
    \langle
      E_\infty(z, w),
      y^{\frac{1}{2}} \theta(z)^{2}
      \mu_{j, \mathfrak{a}}^{-1}(z)
    \rangle
    \\
    &=
    \sum_{n \geq 1} r_{2}(n)
    \overline{\rho_{j, \mathfrak{a}}^{-1}(-n)}
    \int_0^\infty
    y^{w - \frac{1}{2}}
    e^{- 2 \pi n y}
    W_{\frac{1}{2}, it_j}(4 \pi n y) \frac{dy}{y}
    \\
    &=
    \frac{%
      \Gamma(w + it_j)
      \Gamma(w - it_j)
    }{%
      (4 \pi)^{w - \frac{1}{2}}
      \Gamma(w)
    }
    \sum_{n \geq 1}
    \frac{%
      r_{2}(n)
      \overline{\rho_{j, \mathfrak{a}}^{-1}(-n)}
    }{n^{w - \frac{1}{2}}},
\end{align*}
where we use~\cite[7.621.3]{GradshteynRyzhik07} to evaluate the integral.

When $\mathfrak{a} = \frac{1}{2}$, the behavior is slightly different as
$\frac{1}{2}$ is a non-singular cusp. We first note that the stabilizer
group $\Gamma_{1/2} \subset \Gamma_0(N)$ is generated by
\begin{equation*}
	 \gamma_{\frac{1}{2}} :=
  \begin{pmatrix}
    1 & 0 \\ 2 & 1
  \end{pmatrix}
  \begin{pmatrix}
    1 & 1 \\ 0 & 1
  \end{pmatrix}
  \begin{pmatrix}
    1 & 0 \\ 2 & 1
  \end{pmatrix}^{-1}
  =
  \begin{pmatrix}
    -1 & 1 \\ -4 & 3
  \end{pmatrix}.
\end{equation*}
Then $J(\gamma_{1/2}, z) = i$, so $\mu_j^{-1} \vert_{[\sigma_{1/2}^{-1}]}^{-1}$ has period $2$.
These forms have Fourier--Whittaker expansions of the form
\begin{equation*}
  \mu_{j, \frac{1}{2}}^{-1}
  :=
  \mu_j^{-1} \big \vert^{-1}_{[\sigma_{\frac{1}{2}}^{-1}]}
  =
  \sum_{m \in \mathbb{Z}}
  \rho_{j, \frac{1}{2}}^{-1}(2m + 1)
  W_{\frac{-(2m+1)}{2\lvert 2m+1 \rvert}, it_j} (2 \pi \lvert 2m + 1 \rvert y)
  e\big( \tfrac{(2m + 1)x}{2} \big).
\end{equation*}
To see this, note first that $g(z) := \mu_j^{-1}
\vert_{[\sigma_{1/2}^{-1}]}^{-1}(2z)$ has period $1$
and that
\begin{equation}
	\mu_j^{-1}\vert_{[\sigma_{1/2}^{-1}]}^{-1}(z + 1)
		= -1 \cdot \mu_j^{-1}\vert_{[\sigma_{1/2}^{-1}]}^{-1}(z),
\end{equation}
hence $g(z + \frac{1}{2}) = - g(z)$;
then observe in the Fourier--Whittaker expansion of $g(z)$ that this forces the even-indexed coefficients to vanish.

Finally, we also study the behavior of $(\Im(z)^{1/2} \theta(z)^{2})
\big \vert^1_{[\sigma_{1/2}^{-1}]}$. Repeated application of the involution
$\theta(-1/4z) = (-2 i z)^{1/2} \theta(z)$ and the combinatorial
identity $\theta(\frac{1}{2} + z) = 2 \theta(4z) - \theta(z)$ shows that
\begin{align}
\begin{split} \label{eq:theta-1/2-cusp-transformation}
  \theta\bigg( \begin{pmatrix} 1 & 0 \\ 2 & 1 \end{pmatrix}^{-1} z \bigg)
  &= \theta\left(\frac{z}{-2z + 1}\right)
  =
  (-i + \tfrac{i}{2z})^{\frac{1}{2}}
  \theta\left(\frac{1}{2} - \frac{1}{4z}\right)
  \\
  &=
  (-i + \tfrac{i}{2z})^{\frac{1}{2}}
  (-2iz)^{\frac{1}{2}}
  \Big(\theta \left(\frac{z}{4}\right) - \theta(z) \Big)
  \\
  &=
  (-2z + 1)^{\frac{1}{2}}
  \Big(\theta \left(\frac{z}{4}\right) - \theta(z) \Big).
\end{split}
\end{align}
Here, each square root indicates the principal square root, and the final
equality follows by verifying that the square roots may be combined consistently for $z \in \mathcal{H}$.
Further, we see that
\begin{equation} \label{eq:r1o-expansion}
  \theta(\tfrac{z}{4}) - \theta(z)
  =
  \sum_{\substack{m \in \mathbb{Z} \\ m \; \text{odd}}}
  e\left(\frac{m^2 z}{4}\right)
  =:
  \sum_{m \geq 0} r_1^o(m) e\left(\frac{mz}{4}\right),
\end{equation}
where $r_\ell^o(m)$ is the number of representations of $m$ as a sum of $\ell$
\emph{odd} squares, i.e.\ $r_\ell^o(m) := \# \{ \vec{x} \in
(\mathbb{Z}\setminus (2\mathbb{Z}))^\ell: \vec{x} \cdot \vec{x} = m \}$.
Note also that
\begin{equation} \label{eq:r2o-expansion}
  \bigg(\sum_{\substack{m \in \mathbb{Z} \\ m \; \text{odd}}}
  e\left(\frac{m^2 z}{4}\right)\bigg)^2
  =
  \sum_{m \geq 0} r_2^o(4m + 2) e\left(\frac{(2m + 1)z}{2}\right),
\end{equation}
as sums of two odd squares are necessarily $2 \bmod 4$.
It follows that $y^{1/2} \theta^2 \vert_{[\sigma_{1/2}^{-1}]}^{1}$ has period $2$ and Fourier--Whittaker expansion
\begin{equation*}
  (\Im(z)^{\frac{1}{2}} \theta^2) \big
  \vert^1_{[\sigma_{\frac{1}{2}}^{-1}]}
  =
  \Im(z)^{\frac{1}{2}} \sum_{m \geq 0} r_2^o(4m + 2) e\left(\frac{(2m + 1) z}{2}\right).
\end{equation*}

\begin{remark}
  See~\cite[\S{2.7}]{Iwaniec97} for a slightly different method of examining
  the shape of the Fourier expansion at other cusps.
\end{remark}

With these expansions, we compute the inner product~\eqref{eq:Eis_cusp_is_exp_at_cusp} at the $\mathfrak{a} = 1/2$ cusp to be
\begin{align*}
  \langle
    y^{\frac{1}{2}} \overline{\theta^{2}} &E_{\frac{1}{2}}(z, w),
    \mu_j^{-1}
  \rangle
  =
  \langle
    E_\infty(z, w),
    y^{\frac{1}{2}} \theta(z)^{2}
    \mu_{j, \frac{1}{2}}^{-1}
  \rangle
  \\
  &=
  \frac{%
    \Gamma(w + it_j)
    \Gamma(w - it_j)
  }{%
    (2\pi)^{w - \frac{1}{2}}
    \Gamma(w)
  }
  \sum_{n \geq 1}
  \frac{%
    r_{2}^o(4n + 2)
    \overline{\rho_{j, \frac{1}{2}}^{-1}(-2n - 1)}
  }{(2n + 1)^{w - \frac{1}{2}}}.
\end{align*}

We can simplify the Dirichlet series that appears with a sequence of simple observations.
Note that $r_2^o(4n + 2) = r_2(4n + 2)$, as a sum of squares is $2 \bmod 4$ exactly when
both individual squares are odd.
Furthermore, since $2(a^2 + b^2) = (a + b)^2 + (a - b)^2$, we have the classical identity $r_2(2n) =
r_2(n)$, and thus $r_2(4n+2) = r_2(2n + 1)$.
Together, these allow us to replace $r_2^o$ with $r_2$.
Finally, we define $\rho_{j, 1/2}^{-1}(n) = 0$ unless $n \equiv 1 \bmod 2$, so that the Dirichlet series given above better resembles the Rankin--Selberg type Dirichlet series in the cases $\mathfrak{a} = 0$ or $\mathfrak{a} = \infty$.

We collect these computations in the following lemma.

\begin{lemma}\label{lemma:E_a-unfolding}
  For $\Re w \gg 1$, we have that
  \begin{equation}
  \begin{split}
    \langle
      &y^{\frac{1}{2}} \overline{\theta^{2}(z)} E_{\mathfrak{a}}(z, w),
      \mu_j^{-1}(z)
    \rangle
   =
    \frac{%
      \Gamma(w + it_j)
      \Gamma(w - it_j)
    }{%
      (4 \pi/c_{\mathfrak{a}})^{w - \frac{1}{2}}
      \Gamma(w)
    }
    \sum_{n \geq 1}
    \frac{%
      r_{2}(n)
      \overline{\rho_{j, \mathfrak{a}}^{-1}(-n)}
    }{n^{w - \frac{1}{2}}},
  \end{split}
  \end{equation}
  where $c_{\mathfrak{a}} = 2$ if $\mathfrak{a} = \frac{1}{2}$ and is $1$ otherwise.
\end{lemma}

\subsubsection{Recognizing Dirichlet series}

The three Dirichlet series in Lemma~\ref{lemma:E_a-unfolding} are each Rankin--Selberg convolutions
of $r_2(n)$ and $\rho_{j, \mathfrak{a}}^{-1}(-n)$, and we can directly recognize the Dirichlet
series in terms of standard $L$-functions.

Let $\mu^k(z)$ denote any weight $k$ Maass form which is an eigenform under the Hecke operators. If $\mu(z)$ has Fourier--Whittaker coefficients $\rho^k(n)$, then $\rho^k(\pm n) = \rho^k(\pm 1) \lambda^k(\vert n\vert) / \sqrt{n}$, where $\lambda^k(\vert n \vert)$ are the Hecke eigenvalues. The Hecke eigenvalues are multiplicative. For any Dirichlet character $\psi$, we define the $L$-functions $L(s,\mu^k)$ and $L(s,\mu^k \times \psi)$ by
\begin{equation}
	L(s,\mu^k) = \sum_{n \geq 1} \frac{\lambda^k(n)}{n^s},
	\qquad
	L(s,\mu^k \times \psi) = \sum_{n \geq 1} \frac{\lambda^k(n) \psi(n)}{n^s}.
\end{equation}

Lastly, recall that $\frac{1}{4} r_2(n)  = \sum_{d \mid n} \chi(d)$, in which $\chi$ is the non-principal character mod $4$; in particular, $\frac{1}{4} r_2(n)$ is multiplicative.
\begin{lemma}\label{lem:r2rho_rankin_selberg}
  For $\Re w \gg 1$,
  \begin{equation*}
    \sum_{n \geq 1} \frac{r_2(n) \overline{\rho_{j, \mathfrak{a}}^{-1}(-n)}}{n^{w - \frac{1}{2}}}
    =
    4\overline{\rho_{j, \mathfrak{a}}^{-1}(-1)}
    \frac{%
      L(w, \overline{\mu_{j, \mathfrak{a}}^{-1}})
      L(w, \overline{\mu_{j, \mathfrak{a}}^{-1}} \times \chi)
    }{\zeta(2w)},
  \end{equation*}
\end{lemma}

\begin{proof}
  Write the Dirichlet series as
  \begin{equation*}
    \sum_{n \geq 1} \frac{r_2(n) \overline{\rho_{j, \mathfrak{a}}^{-1}(-n)}}{n^{w - \frac{1}{2}}}
    =
    4 \overline{\rho_{j, \mathfrak{a}}^{-1}(-1)}
    \sum_{n \geq 1}
    \frac{\frac{1}{4}r_2(n) \cdot \overline{\lambda_{j, \mathfrak{a}}^{-1}(n)}}{n^{w}}.
  \end{equation*}
  Directly comparing Euler products completes the proof.
\end{proof}

We now assemble the components of $\Sigma_{\mathrm{disc}}$.
Combining the decomposition from Proposition~\ref{prop:Elevel1_is_Elevel4_sum} (after multiplying
through by the Eisenstein completing factor $\zeta^*(2w)$), the inner product evaluation from
Lemma~\ref{lemma:E_a-unfolding}, and the Dirichlet series evaluation from
Lemma~\ref{lem:r2rho_rankin_selberg}, we produce the following.

\begin{proposition}\label{prop:discrete-k2}
  The discrete spectral contribution towards $D_h(s, w)$ can be written as
  \begin{align*}
    \Sigma_{\mathrm{disc}}
    &=
    \frac{8}{h^{s - 1}(2\pi)^{2w - 1}}
    \sum_j \frac{\Gamma(s-\tfrac{1}{2}+it_j)\Gamma(s-\frac{1}{2}-it_j)}
                {\Gamma(s-\frac{1}{2}+w)\Gamma(s+\frac{1}{2}-w)}
    \rho_{j,\infty}^{-1}(h)
    \\
    &\quad \times \Gamma(w+it_j)\Gamma(w-it_j)
    \sum_{\mathfrak{a}} b_{\mathfrak{a}}^w c_{\mathfrak{a}}^{w - \frac{1}{2}}
    \overline{\rho^{-1}_{j,\mathfrak{a}}(-1)}
    L(w, \overline{\mu_{j, \mathfrak{a}}^{-1}}) L(w, \overline{\mu_{j, \mathfrak{a}}^{-1}} \times \chi),
  \end{align*}
  in which $b_\infty = b_\frac{1}{2} = 1$ and $b_0 = 4$, and $c_{\infty} = c_{0} = 1$ and $c_{\frac{1}{2}}
  = 2$.
\end{proposition}

\begin{remark}\label{remark:discrete_k2_has_continuation_in_w}
  We note that this expression has clear meromorphic continuation in \emph{both} $s$ and $w$, wherever the sum converges, from
  the meromorphic continuations of the individual gamma and $L$-functions.
\end{remark}

\subsubsection{Bounding Maass form coefficients}

To conclude our meromorphic description of $\Sigma_{\mathrm{disc}}$, we need to understand the
Maass form coefficients $\rho_{j, \mathfrak{a}}^{-1}$ that appear in
Proposition~\ref{prop:discrete-k2}.

Varied techniques have been used to study the sizes of integral
weight Maass form coefficients.
A standard application of Proskurin's Kuznetsov trace formula in weight $1$ shows the on-average bound
\begin{equation} \label{eq:kuznetsov-k2}
  \sum_{\frac{T}{2} \leq \lvert t_j \rvert < 2T}
  \frac{\lvert \rho_{j, \mathfrak{a}}^{-1}(h) \rvert^2}{\cosh \pi t_j}
  \ll_{h, \mathfrak{a}, \epsilon}
  T^{2 + \sgn{(h)}}.
\end{equation}
Stirling's formula and the crude union bound
$\lvert \rho_{j,\mathfrak{a}}^{-1}(-1) \rvert^2 \ll \cosh(\pi t_j) \lvert t_j \rvert^{1+\epsilon}$
suffice to show that the expansion for $\Sigma_{\mathrm{disc}}$ in
Proposition~\ref{prop:discrete-k2} has exponential decay in $t_j$ and converges absolutely for any
fixed $s$ and $w$ away from the poles of its gamma function factors.
This is sufficient to establish the meromorphic continuation of $\Sigma_{\mathrm{disc}}$.
Sharper bounds related to the discrete spectrum will be considered in~\S\ref{sec:growth_discrete}.

For later use, we record here a stronger version of~\eqref{eq:kuznetsov-k2} which averages over the short interval $T \leq \vert t_j \vert \leq T+1$. We restrict to $\mathfrak{a} = \infty$ and $h > 0$, as this suffices for our application. Other cases may be treated similarly.

\begin{lemma} \label{lem:Kuznetsov-bound}
Fix $h > 0$. As $T \to \infty$, we have
\begin{equation}
  \sum_{T \leq \lvert t_j \rvert \leq T+1}
		\frac{\lvert \rho_{j}^{-1}(h) \rvert^2}{\cosh \pi t_j}
  \ll_{h} T^{2}.
\end{equation}
\end{lemma}

\begin{proof}
Fix $T \gg 1$ and define the test function
\[
	g(t,T) := \frac{\cosh \pi t \cosh \pi T}{\cosh \pi (t-T) \cosh \pi (t+T)},
\]
which concentrates in the region $t = \pm T + O(1)$.
The Kuznetsov trace formula in weight $-1$, as presented in~\cite[Proposition 5.2]{DFIsub}, implies that
\begin{align} \label{eq:simplified-Kuznetsov}
    & \sum_j \frac{\vert \rho_j^{-1}(h) \vert^2 g(t_j,T)}{\cosh \pi t_j}
		+ \sum_{\mathfrak{a}} \frac{1}{4\pi} \int_{\mathbb{R}}
			\frac{\vert \rho_\mathfrak{a}^{-1}(h,\tfrac{1}{2}+it) \vert^2 g(t,T)}
						{\cosh \pi t} \,dt \\
    & \qquad
		= \frac{(4T^2+1)}{16\pi^2 h} \bigg(1 + 8\pi h
		\sum_{c \equiv 0 (4)} \frac{S(h,h;c)}{c^2}
		\int_{-i}^i \! K_{2iT}\Big(\frac{4\pi h}{c}\zeta \Big) \frac{d\zeta}{\zeta^2} \bigg),
\end{align}
in which the $\zeta$-integral runs counter-clockwise along the right half of the unit circle and $S(m,n; c)$ denotes the Kloosterman sum.

Equation~\eqref{eq:simplified-Kuznetsov}, positivity in the continuous spectrum, and the lower bound $g(t,T) \gg 1$ for $t = \pm T + O(1)$ implies that
\[
	\sum_{T \leq \vert t_j \vert \leq T+1}
		\frac{\vert \rho_j^{-1}(h) \vert^2}{\cosh \pi t_j}
	\ll \frac{T^2}{h}
		+ T^2 \! \sum_{c \equiv 0(4)} \! \frac{\vert S(h,h;c) \vert}{c^2}
			\cdot \bigg\vert \int_{-i}^i K_{2iT} \Big(\frac{4\pi h}{c}\zeta \Big) \frac{d\zeta}{\zeta^2}
			\bigg\vert.
\]

Let $\beta = \frac{4\pi h}{c}$ and let $I_T(\beta)$ denote the contour integral
above. By changing variables in~\cite[10.32.14]{DLMF}, we produce the integral
representation
\begin{equation} \label{eq:K-Bessel-representation}
	K_{2iT}(z) =
	\frac{1}{4\pi i} \int_{(\sigma)} \Gamma(u-iT)\Gamma(u+iT) \Big(\frac{z}{2}\Big)^{-2u} \, du,
\end{equation}
valid for $\sigma > 0$ and $\vert \mathrm{arg} z \vert < \frac{\pi}{2}$. For later use, we suppose that $\sigma = \frac{1}{4} - \epsilon$. By truncating the contour of $I_r(\beta)$ to $\vert \mathrm{arg} \zeta \vert \leq \frac{\pi}{2} - \delta$, replacing $K_{2iT}(\beta \zeta)$ with~\eqref{eq:K-Bessel-representation}, reversing the order of integration, and tending $\delta \to 0$, we produce
\begin{align*}
	I_T(\beta) &=
	\frac{1}{4\pi i} \int_{(\frac{1}{4}-\epsilon)}
		\Gamma(u-iT) \Gamma(u+iT) \Big(\frac{\beta}{2}\Big)^{-2u}
		\Big( \int_{-i}^i \zeta^{-2u -2} d\zeta \Big) du \\
	&= 	\frac{1}{4\pi i} \int_{(\frac{1}{4}-\epsilon)}
		\Gamma(u-iT) \Gamma(u+iT) \Big(\frac{\beta}{2}\Big)^{-2u}
		\cdot \Big(\frac{2i \cos(\pi u)}{2u+1} \Big) du.
\end{align*}
Bounding the integral using absolute values and Stirling's approximation produces $I_T(\beta) \ll_\epsilon \beta^{-\frac{1}{2} + 2\epsilon} T^{-\frac{1}{2}}$. It follows that
\[
	\sum_{T \leq \vert t_j \vert \leq T+1}
		\frac{\vert \rho_j^{-1}(h) \vert^2}{\cosh \pi t_j}
	\ll \frac{T^2}{h}
		+ T^\frac{3}{2} h^{-\frac{1}{2} + 2 \epsilon}
			\sum_{c \geq 1} \frac{\vert S(h,h;c) \vert}{c^{3/2+\epsilon}}.
\]
Applying either the Weil bound or~\cite[(16.50)]{IwaniecKowalski04} to handle the
Kloosterman sums shows this is $O_h(T^2)$.
\end{proof}

\subsection{Continuation of the non-spectral terms}\label{ssec:nspec_continuation}

For $w \neq \frac{1}{2}$, we may write $\Sigma_{\mathrm{reg}} = \mathfrak{E}_+ +
\mathfrak{E}_-$, where
\begin{align*}
  \mathfrak{E}_+(s,w)
  :=&
  \frac{(4\pi)^{\frac{1}{2}}\zeta^*(2w)\Gamma(s-w-\frac{1}{2})}
       {h^{s-1}\Gamma(s+\frac{1}{2}-w)}
  \varphi_h(\tfrac{1}{2}+w),
  \\
  \mathfrak{E}_-(s,w)
  :=&
  \frac{(4\pi)^{\frac{1}{2}}\zeta^*(2-2w) \Gamma(s+w-\frac{3}{2})}
       {h^{s-1}\Gamma(s-\frac{1}{2}+w)}
  \varphi_h(\tfrac{3}{2}-w)
\end{align*}
denote the two regularizing terms which appear in Proposition~\ref{prop:D_h-spectral}.
The Eisenstein coefficients $\varphi_h(\frac{1}{2}+w)$ have potential poles at zeros of
$L(2w+1,\chi)$ (cf.\ Lemma~\ref{lem:Eisenstein-series-coefficients-formulas-wt1}), which we avoid
under the assumption $0<\Re w < 1$.

We observe that the apparent pole at $w = \frac{1}{2}$ is removable.
The obvious functional equation $\mathfrak{E}_-(s,w) =
\mathfrak{E}_+(s,1-w)$ implies that the potential poles from the zeta functions in the numerators of
$\mathfrak{E}_-$ and $\mathfrak{E}_+$ cancel, and we compute that
\begin{align}
\label{eq:non-spectral-w=1/2}
  \lim_{w \to \frac{1}{2}}
  \big(\mathfrak{E}_+(s,w) + \mathfrak{E}_-(s,w)\big)
  &=
  \frac{(4\pi h)^{\frac{1}{2}} \Gamma(s-1)}
       {h^{s-\frac{1}{2}} \Gamma(s)}
  \Big(%
    \varphi_h'(1)
    \\
    &\quad + \big(%
      \psi(s) - \psi(s-1)
      + \gamma - \log (4\pi)
    \big) \varphi_h(1)
  \Big),
  \notag{}
\end{align}
where $\psi(s) = \Gamma'(s)/\Gamma(s)$ denotes the digamma function.

This function has a double pole at $s=1$ and is otherwise holomorphic in $\Re s > 0$. Thus
$\Sigma_{\mathrm{reg}}$, originally defined for $\Re w \in (0, 1)$ with $w \neq \frac{1}{2}$, has a clear meromorphic
continuation to $w = \frac{1}{2}$.
Further, as each of $\Sigma_{\mathrm{disc}}$ and $\Sigma_{\mathrm{cont}}$ have clear meromorphic
continuation to $w = \frac{1}{2}$, we recognize that $D_h(s, w)$ also continues to $w =
\frac{1}{2}$.

The non-spectral terms are the source of the rightmost poles of $D_h(s,w)$. For later
applications, we record the explicit meromorphic behavior of
$\Sigma_{\mathrm{reg}}$ in the right
half-plane $\Re s > 0$.

Note that most potential poles from the gamma functions in the numerators
cancel with poles from the gamma functions in the denominators.
When $w \neq \frac{1}{2}$, $\Sigma_{\mathrm{reg}}$ has a pole
at $s = w + \frac{1}{2}$ from $\mathfrak{E}_+$ and a pole at $s = \frac{3}{2} -
w$, coming from $\mathfrak{E}_-$.
These poles have residues
\begin{align*}
  \Res_{s = \frac{1}{2} + w} \mathfrak{E}_+(s, w)
  &=
  \frac{(4\pi)^{\frac{1}{2}} \zeta^*(2w)}{h^{w - \frac{1}{2}}} \varphi_h(w
  + \tfrac{1}{2}) \qquad \qquad (w \neq \tfrac{1}{2}),
  \\
  \Res_{s = \frac{3}{2} - w} \mathfrak{E}_-(s, w)
  &=
  \frac{(4\pi)^{\frac{1}{2}} \zeta^*(2-2w)}{h^{\frac{1}{2} - w}}
  \varphi_h(\tfrac{3}{2} - w)
  \qquad (w \neq \tfrac{1}{2}).
\end{align*}
There are no other poles in $\Re s > 0$.
When $w = \frac{1}{2}$, observe that~\eqref{eq:non-spectral-w=1/2} simplifies to
\begin{equation*}
  \frac{(4\pi h)^{\frac{1}{2}}}{h^{s - \frac{1}{2}}}
  \left(
    \frac{\varphi_h(1)}{(s-1)^2}
    + \frac{\varphi_h'(1) + (\gamma - \log(4 \pi h)) \varphi_h(1)}{s - 1}
  \right),
\end{equation*}
so $\Sigma_{\mathrm{reg}}$ admits a double pole at $s = 1$ and is otherwise
holomorphic in $s$.

\subsection*{Meromorphic continuation}

We have now proved the continuations and convergence of the discrete, continuous, and nonspectral
components.
Combining yields Theorem~\ref{thm:D_h(s,w)-poles}.

\begin{remark} \label{rem:general-w}
Once the meromorphic continuation of $D_h(s,w)$ is established for $s \in \mathbb{C}$ and $\Re w \in (0,1)$, it can be extended to all $w \in \mathbb{C}$ through meromorphic continuation within the individual terms $\Sigma_{\mathrm{cont}}$, $\Sigma_{\mathrm{disc}}$, and $\Sigma_{\mathrm{reg}}$. The continuation of $\Sigma_{\mathrm{cont}}$ is non-obvious, as this term required $\Re w \in (0,1)$ initially to justify the convergence of $\langle V_w, E_\mathfrak{a}^{-1} \rangle$. However, as seen in Proposition~\ref{prop:V-E_a-inner-product}, the inner product $\langle V_w, E_\mathfrak{a}^{-1} \rangle$ has explicit meromorphic continuation in $w$, so that $\Sigma_{\mathrm{cont}}(s,w)$ and thus $D_h(s,w)$ extend to meromorphic functions in $(s,w) \in \mathbb{C}^2$. We suppress this generality, as Corollary~\ref{cor:cor_introduction} requires information about $D_h(s,w)$ only in a neighborhood of $w=\frac{1}{2}$.
\end{remark}

\section{Growth of \texorpdfstring{$D_h(s,w)$}{Dh(s,w)} in Vertical Strips}\label{sec:Dh(s,w)-growth}

Theorem~\ref{thm:D_h(s,w)-poles} gives information about the rightmost poles of the meromorphic
continuation of $D_h(s, w)$.
To study partial sums of the coefficients of $D_h(s, w)$, we must understand the growth of $D_h(s,w)$
with respect to $\lvert \Im s \rvert$ as well.
In this section, we show that $D_h(s, w)$ has polynomial growth in $\lvert \Im s \rvert$ within
vertical strips, for fixed $w$ in the region $0 < \Re w < 1$.

To reduce casework in our proof, we leverage the functional equation $D_h(s,w) = D_h(s,1-w)$ (which
comes directly from the functional equation $\sigma_{\nu}(n) = n^\nu \sigma_{-\nu}(n)$) to assume throughout
that $\Re w \geq \frac{1}{2}$.
We then quantify the growth in $\lvert \Im s \rvert$ by treating $\Sigma_{\mathrm{cont}}$,
$\Sigma_{\mathrm{disc}}$, and $\Sigma_{\mathrm{reg}}$ separately.

Our analysis is most complicated for the discrete spectrum $\Sigma_{\mathrm{disc}}$,
which dominates the growth estimates for $\Sigma_{\mathrm{cont}}$ and
$\Sigma_{\mathrm{reg}}$ and thus represents the obstruction towards improvement.
We provide unconditional bounds coming from current results (as given in
Appendix~\ref{app:huang_kuan}), as well as conditional bounds which assume the spectral fourth moment
result for weight $1$ Maass forms (Conjecture~\ref{conj:intro-spectral-fourth-moment}).

The growth bounds for $\Sigma_{\mathrm{cont}}$ and
$\Sigma_{\mathrm{reg}}$ in this section restrict to the half-plane $\Re s > \frac{1}{2}$, as this suffices for our
main application.
We require more information about $\Sigma_{\mathrm{disc}}$ and thus bound it for general $s$.
The methods given here could describe $\Sigma_{\mathrm{cont}}$ and
$\Sigma_{\mathrm{reg}}$ for general $s$ if needed. We will prove the following theorem.

\begin{theorem}\label{thm:Dh(s,w)-growth}
  For any $w$ in the vertical strip $\Re w \in (0, 1)$, any $s$ with $\Re s > \frac{1}{2}$, and any
  $\epsilon > 0$, we have
  \begin{equation}
		D_h(s,w)
    \ll_{h,w, \epsilon}
    \lvert s \rvert^{\frac{3}{2} + \lvert \Re w - \frac{1}{2}\rvert + \epsilon}
  \end{equation}
  as $\lvert s \rvert \to \infty$ in a fixed vertical strip.
  With Conjecture~\ref{conj:intro-spectral-fourth-moment}, this improves to
  \begin{equation*}
    D_h(s,w)
    \ll_{h,w, \epsilon}
    \lvert s \rvert^{1 + 2\lvert \Re w-\frac{1}{2} \rvert + \epsilon}
  \end{equation*}
  as $\lvert s \rvert \to \infty$ in a fixed vertical strip.
\end{theorem}

This bound for the total growth of $D_h(s,w)$ in vertical strips will follow from
Proposition~\ref{prop:continuous-growth} (growth in $\Sigma_{\mathrm{cont}}$),
Proposition~\ref{prop:ns-growth} (growth in
$\Sigma_{\mathrm{reg}}$), and
Proposition~\ref{prop:ds-growth-k2} (growth in $\Sigma_{\mathrm{disc}}$).
We study each in turn.

\subsection{Growth in \texorpdfstring{$\Sigma_{\mathrm{cont}}$}{the continuous spectrum}}

In the region $\Re s > \frac{1}{2}$, the first residual term~\eqref{eq:first_residual} does not
appear.
To study growth in $\Sigma_{\mathrm{cont}}$, it therefore suffices to study only the
integral~\eqref{eq:continuous-spectrum}.
We first produce estimates for
$\rho_{\mathfrak{a}}^{-1}(h, u) \langle V_w, E_\mathfrak{a}^{-1} \rangle$ on the critical line
$\Re u = \frac{1}{2}$.

\begin{lemma}\label{lem:Eisenstein-coefficient-inner-product-bound}
  Let $\mathcal{L}_w(t)$ denote the collected $L$-functions
  \begin{equation}
    \mathcal{L}_w(t)
    =
      L(w - it, \chi)
      L(w + it, \chi)
      \zeta(w - it)
      \zeta(w + it).
  \end{equation}
  Suppose $h \in \mathbb{Z}_{>0}$ and let $\mathfrak{a}$ denote any cusp of $\Gamma_0(4)$.
  For any $w$ with $\frac{1}{2} \leq \Re w < 1$, we have as $\lvert t \rvert \to \infty$ that
  \begin{equation}
    \rho_{\mathfrak{a}}^{-1}(h, \tfrac{1}{2}+it)
    \big\langle
      V_w,
      E_\mathfrak{a}^{-1}(\cdot, \tfrac{1}{2}+it)
    \big\rangle
    \ll_{h,w}
    (\log \lvert t \rvert)^2 \lvert t \rvert^{2 \Re w - \frac{3}{2}}
    \lvert \mathcal{L}_w(t) \rvert.
  \end{equation}
\end{lemma}

\begin{proof}
The four $L$-functions in $\mathcal{L}(t)$ arise from the
collected $L$-functions of the coefficients $\rho_{\mathfrak{a}}^{-1}$ (as
given in Lemma~\ref{lem:Eisenstein-series-coefficients-formulas-wt1}) and the inner products
$V_w$ against $E_\mathfrak{a}^{-1}$ (as given in Proposition~\ref{prop:V-E_a-inner-product}).

It is convenient to rewrite the $L$-functions from Proposition~\ref{prop:V-E_a-inner-product} using
$\zeta^*(1 - w + it) = \zeta^*(w - it)$ and $L^*(1 - w + it, \chi) = L^*(w - it, \chi)$ to form
$\mathcal{L}_w(t)$.
We then have
\begin{align*}
  \rho_{\mathfrak{a}}^{-1}(h,\tfrac{1}{2}+it)
  &\big\langle V_w, E_\mathfrak{a}^{-1}(\cdot, \tfrac{1}{2}+it)\big\rangle
  \\
  & \ll_{h,w}
  \frac{%
    \Gamma(\frac{w - it}{2})^2 \Gamma(\frac{w + it}{2})^2
  }{%
    \Gamma(it) \Gamma(\frac{1}{2} + it) L(1+2it,\chi)^2
  }
  \mathcal{L}_w(t),
\end{align*}
whereby Stirling's approximation and the classical estimate $1/L(1+it,\chi) \ll \log \lvert t \rvert$
(see~\cite[(11.6)]{MontgomeryVaughn06}) complete the proof.
\end{proof}

Recall from~\eqref{eq:continuous-spectrum} that $\Sigma_{\mathrm{cont}}$ is analytic for $\Re s >
\frac{1}{2}$, where it equals
\begin{equation}
  \begin{split}
    \Sigma_{\mathrm{cont}}(s,w)
    =
    \frac{(4\pi)^{-\frac{1}{2}}} {h^{s-1} i}
    \sum_{\mathfrak{a}} \! \int_{(0)}
    & \frac{\Gamma(s-\frac{1}{2}+u)\Gamma(s-\frac{1}{2}-u)}
         {\Gamma(s-\frac{1}{2}+w)\Gamma(s+\frac{1}{2}-w)}
    \\
    & \times
    \rho_{\mathfrak{a}}^{-1}(h, \tfrac{1}{2} +u)
    \langle
      V_w,
      E_{\mathfrak{a}}^{-1}(\cdot,\tfrac{1}{2} + \overline{u})
    \rangle du.
  \end{split}
\end{equation}
The previous lemma and Stirling's approximation allows one to bound $\Sigma_{\mathrm{cont}}$ in
terms of moments of classical $L$-functions.
Bounds of this type are common in the literature, and by leveraging known results we produce the
following.

\begin{proposition}\label{prop:continuous-growth}
  Fix $w$ with $\Re w \in [\frac{1}{2}, 1)$.
  For $\Re s > \frac{1}{2}$ and any $\epsilon > 0$, the contribution towards $D_h(s,w)$ from the
  continuous spectrum satisfies
  \begin{equation*}
    \Sigma_{\mathrm{cont}}(s,w)
    \ll_{h, w, \epsilon}
    \lvert s \rvert^{-\frac{1}{6} + \frac{4}{3} ( \Re w - \frac{1}{2} ) + \epsilon}
  \end{equation*}
  as $\lvert \Im s \rvert \to \infty$.
\end{proposition}

\begin{proof}
Write $\sigma = \Re s$ and $u = it$.
Standard convexity estimates show that $\mathcal{L}_w(t)$ grows at most polynomially in $t$.
Stirling's approximation then shows that for $\lvert t \rvert > \lvert s \rvert$, the integrand
exponentially decays and the mass of the integrand concentrates in the region $\lvert t \rvert <
\lvert s \rvert^{1 + \epsilon}$.

We apply H\"older's inequality with exponents $(\frac{1}{2}, \frac{1}{4}, \frac{1}{4})$ to show that
the contribution of the continuous spectrum over the interval
$\lvert t \rvert < \lvert s \rvert^{1 + \epsilon}$ is bounded by
\begin{equation}\label{eq:cauchy-schwarz-x2-k2}
\begin{split}
  \ll_{h,w}
  \Bigg(
    &\int_1^{\lvert s \rvert^{1 + \epsilon}}
    \frac{(\log t)^4 t^{4 \Re w}}{\lvert s \rvert^{4\sigma-2} t^3}
    \frac{\lvert \zeta(w + it)\zeta(w - it)\rvert^2}
         {(1 + \lvert s+t \rvert)^{2-2\sigma}
          (1 + \lvert s-t \rvert)^{2-2\sigma}} dt
  \Bigg)^{\frac{1}{2}}
  \\
  & \times
  \Bigg(
    \int_1^{\lvert s \rvert^{1 + \epsilon}} \lvert L(w+it,\chi) \rvert^4 dt
  \Bigg)^{\frac{1}{4}}
  \Bigg(
    \int_1^{\lvert s \rvert^{1 + \epsilon}} \lvert L(w-it,\chi) \rvert^4 dt
  \Bigg)^{\frac{1}{4}}.
\end{split}
\end{equation}
For $\frac{1}{2} \leq \Re w < 1$, the classical critical line estimate $\zeta(\frac{1}{2}+it) \ll (1+ \vert t \vert)^{1/6+\epsilon}$ and convexity principle as above imply
that
\begin{equation}
  \lvert \zeta(w + it) \zeta(w - it) \rvert^2
  \ll (1 + \lvert t \rvert)^{\frac{4}{3}-\frac{4}{3}\Re w + \epsilon}.
\end{equation}
A short computation shows that the first line of~\eqref{eq:cauchy-schwarz-x2-k2} is
$O(\lvert s \rvert^{\frac{4}{3} \Re w -\frac{4}{3}  + \epsilon})$ when $\frac{1}{2} \leq \Re w < 1$.
(This uses $\Re s =\sigma > \frac{1}{2}$).

The second and third terms in~\eqref{eq:cauchy-schwarz-x2-k2} are fourth moments of quadratic
Dirichlet $L$-functions not necessarily on the critical line.
As noted above, the Lindel\"of Hypothesis is known on average in the $t$-aspect on the critical line; in the rest of
the critical strip, an integral convexity argument in the $w$ variable (as shown in~\cite[\S7.8]{Titchmarsh86}, for example) implies that
\begin{equation}
  \int_1^{\lvert s \rvert^{1 + \epsilon}} \lvert L(w+it,\chi) \rvert^4 dt
  \ll
  \vert s \vert^{1 + \epsilon}
\end{equation}
for $\Re w \geq \frac{1}{2}$.
Combining together, it follows that~\eqref{eq:cauchy-schwarz-x2-k2} is bounded by
\begin{equation}
  O\big(\vert s \vert^{\frac{4}{3} \Re w -\frac{4}{3} + \frac{1}{2} + \epsilon}\big)
  =
  O\big(\vert s \vert^{-\frac{1}{6} + \frac{4}{3} (\Re w - \frac{1}{2}) + \epsilon}\big).
\end{equation}
This completes the proof.
\end{proof}

\begin{remark}
  The bounds in Proposition~\ref{prop:continuous-growth} depend on subconvexity estimates for
  $\zeta(s)$ and could be improved using a sharper subconvexity result, such as
  $\zeta(\frac{1}{2}+it) \ll (1+\vert t \vert)^{13/84+\epsilon}$ due to Bourgain~\cite{Bourgain17}.
  Under the Lindel\"of hypothesis, we obtain the improved bound
	$\Sigma_{\mathrm{cont}} \ll \lvert s \rvert^{-\frac{1}{2}+2( \Re w - \frac{1}{2} ) + \epsilon}$.
  Regardless, our bounds for $\Sigma_{\mathrm{cont}}$ are not
  the primary obstruction.
\end{remark}

\subsection{Growth in
\texorpdfstring{$\Sigma_{\mathrm{reg}}$}{the non-spectral terms}}%
\label{sec:growth-non-spectral}

We handle the cases $w \neq \frac{1}{2}$ and $w = \frac{1}{2}$ separately.
For $w \neq \frac{1}{2}$, we write  $\Sigma_{\mathrm{reg}}(s,w) = \mathfrak{E}_+(s,w) +
\mathfrak{E}_-(s, w)$ as in~\S\ref{ssec:nspec_continuation}.
All growth with respect to $\lvert \Im s \rvert$ arises from the ratio of gamma functions.
Stirling's approximation gives
\begin{equation*}
  \mathfrak{E}_+(s,w)
  :=
  \frac{(4\pi h)^{\frac{1}{2}}\zeta^*(2w)\Gamma(s-w-\frac{1}{2})}
       {h^{s-\frac{1}{2}}\Gamma(s+\frac{1}{2}-w)}
  \varphi_h(\tfrac{1}{2}+w)
  \ll_{h,w}
  \lvert s \rvert^{-1}.
\end{equation*}
The same bound holds for $\mathfrak{E}_-(s,w) = \mathfrak{E}_+(s, 1-w)$.

For $w=\frac{1}{2}$, recall from~\eqref{eq:non-spectral-w=1/2} that the non-spectral contribution is
\begin{equation}
  \frac{(4\pi h)^{\frac{1}{2}} \Gamma(s - 1)}
       {h^{s-\frac{1}{2}} \Gamma(s)}
  \Big(%
    \varphi_h'(1)
    +
    \big(%
      \psi(s) - \psi(s-1)
      + \gamma - \log (4\pi)
    \big) \varphi_h(1)
  \Big),
\end{equation}
where $\psi(s) = \Gamma'(s)/\Gamma(s)$ is the digamma function.
The asymptotic expansion $\psi(s)-\psi(s-1) = \frac{1}{s}+O(s^{-2})$ and Stirling's
approximation imply that the non-spectral contribution remains $O_h(\lvert s \rvert^{-1})$.

Together, these cases prove the following.

\begin{proposition}\label{prop:ns-growth}
  Fix $w$ with $\Re w \in (0, 1)$.
  The contribution from the non-spectral portion
  $\Sigma_{\mathrm{reg}}$ of
  $D_h(s,w)$ satisfies the bound
  \begin{equation}
    \Sigma_{\mathrm{reg}}(s,w)
    =
    O_{h,w} \big(
      \lvert s \rvert^{-1}
    \big)
  \end{equation}
  as $\lvert \Im s \rvert \to \infty$ in a vertical strip.
\end{proposition}

\subsection{Growth in \texorpdfstring{$\Sigma_{\mathrm{disc}}$}{the discrete spectrum}}%
\label{sec:growth_discrete}

Finally, we discuss the contribution of the discrete spectrum towards bounds for
$D_h(s,w)$ in vertical strips.

By Proposition~\ref{prop:discrete-k2}, $\Sigma_{\mathrm{disc}}$ is bounded by
\begin{equation}\label{eq:discrete-k2-abs}
\begin{split}
    \Sigma_{\mathrm{disc}}
    \ll_{h, w}
    \sum_j
    \sum_{\mathfrak{a}} &
    \Bigl \lvert
    \frac{\Gamma(s-\tfrac{1}{2}+it_j)\Gamma(s-\frac{1}{2}-it_j)}
         {\Gamma(s-\frac{1}{2}+w)\Gamma(s+\frac{1}{2}-w)}
    \rho_{j,\infty}^{-1}(h)
    \overline{\rho^{-1}_{j,\mathfrak{a}}(-1)}
    \\
    &\times \Gamma(w+it_j)\Gamma(w-it_j)
    L(w, \overline{\mu_{j, \mathfrak{a}}^{-1}}) L(w, \overline{\mu_{j, \mathfrak{a}}^{-1}} \times \chi)
    \Bigr \rvert.
\end{split}
\end{equation}
The Kuznetsov estimate~\eqref{eq:kuznetsov-k2} and Stirling's formula imply that the $j$-sum in
$\Sigma_{\mathrm{disc}}$ is negligible in the range
$\lvert t_j \rvert \geq \lvert s \rvert^{1+\epsilon}$.
To estimate the contribution from $\lvert t_j \rvert \leq \lvert s \rvert^{1+\epsilon}$, we apply
Stirling's formula and H\"older's inequality with exponents
$(\frac{1}{2}, \frac{1}{4},\frac{1}{4})$ to produce the upper bound
\begin{align}\label{eq:discrete-holder-k2}
  \Sigma_{\mathrm{disc}} &\ll_{h,w} \sum_{\mathfrak{a}}
  \bigg(%
    \sum_{\vert t_j \vert \leq \vert s \vert^{1+\epsilon}}
    \!\!\!
      \frac{\vert s-it_j \vert^{2(\sigma-1)} \vert s + it_j \vert^{2(\sigma-1)}}
           {\vert s \vert^{4\sigma-2} \vert t_j \vert^{2-4\Re w}}
      \cdot
      \frac{\vert \rho_{j,\infty}^{-1}(h) \vert^2}
           {e^{\pi \vert t_j \vert}}
  \bigg)^\frac{1}{2}
  \\
  & \quad \times
  \bigg(%
    \sum_{\vert t_j \vert \leq \vert s \vert^{1+\epsilon}}
    \alpha_{j,\mathfrak{a}}
    \lvert L(w,\overline{\mu_{j,\mathfrak{a}}^{-1}}) \rvert^4
  \bigg)^\frac{1}{4}
  \bigg(%
    \sum_{\vert t_j \vert \leq \vert s \vert^{1+\epsilon}}
    \alpha_{j,\mathfrak{a}} \lvert L(w,\overline{\mu_{j,\mathfrak{a}}^{-1}} \times \chi) \rvert^4
  \bigg)^\frac{1}{4},
\end{align}
in which $\alpha_{j,\mathfrak{a}} := \lvert \rho_{j,\mathfrak{a}}^{-1}(-1) \rvert^2 / \cosh(\pi t_j)$.

To treat the first line in~\eqref{eq:discrete-holder-k2}, we split the range of
summation into two cases, $\vert t_j \vert \leq \frac{1}{2} \vert s \vert$ and
$\frac{1}{2} \vert s \vert \leq  \vert t_j \vert \leq \vert s \vert^{1+\epsilon}$. The first case contributes
\begin{align*}
	 O_{h,w}\Big( \vert s \vert^{2\Re w -2} \Big(\sum_{\vert t_j \vert \leq \frac{1}{2} \vert s \vert}
	\frac{\vert \rho_{j,\infty}^{-1}(h)\vert^2}{\cosh \pi t_j} \Big)^{\frac{1}{2}} \Big)
	= O_{h,w,\epsilon}\big(\vert s \vert^{2\Re w - \frac{1}{2} + \epsilon} \big)
\end{align*}
via~\eqref{eq:kuznetsov-k2}. For the second case, we assume without loss of generality that $\Im s > 0$. This case contributes
\begin{align} \label{eq:large-tj-case}
	 O_{h,w}\Big( \vert s \vert^{2\Re w -1-\sigma} \Big(\sum_{\frac{1}{2} \vert s \vert \leq t_j \leq \vert s \vert^{1+\epsilon}}
	\vert s - it_j \vert^{2\sigma -2} \frac{\vert \rho_{j,\infty}^{-1}(h)\vert^2}{\cosh \pi t_j} \Big)^{\frac{1}{2}} \Big).
\end{align}

To bound~\eqref{eq:large-tj-case}, we subdivide $\frac{1}{2} \vert s \vert \leq t_j \leq \vert s \vert^{1+\epsilon}$ into sub-intervals of length $1$, on which $\vert s - it_j \vert^{2\sigma -2}$ is slowly varying. By bounding the contribution on each interval using the short-interval Kuznetsov bound from Lemma~\ref{lem:Kuznetsov-bound}, we conclude that~\eqref{eq:large-tj-case} is $O(\vert s \vert^{2\Re w - 1 -\sigma+\epsilon} (\vert s \vert^{\sigma+\frac{1}{2}} + \vert s \vert))$. In particular, the first line of~\eqref{eq:discrete-holder-k2} is
\[
	O_{h,w,\epsilon}\big(
		\vert s \vert^{2\Re w - \frac{1}{2} + \epsilon} (1 + \vert s \vert^{\frac{1}{2}-\sigma})
	\big).
\]
This result assumes that $\Re w \in [\frac{1}{2},1)$ but does not assume anything about $\Re s = \sigma$ except that it lies in a fixed vertical strip.

To treat the terms in the second line of~\eqref{eq:discrete-holder-k2}, we apply spectral fourth
moment results for weight $-1$ Maass forms.
These results are more naturally stated for Maass forms of weight $1$, so we briefly recall the relationship between Maass forms of weight $k$ and $-k$. Let $\mu^k(z)$ be any Maass form of weight $k$, with Fourier--Whittaker expansion
\[
  \mu^{k}(z)
  =
  \sum_{m \neq 0} \rho^{k}(m)
  W_{\frac{mk}{2\lvert m \rvert},it}(4\pi \lvert m \rvert y)e(mx).
\]
Complex conjugation maps the space of weight $k$ Maass forms bijectively to the space of weight $-k$ Maass forms, preserving level, spectral type, and $L^2$ norm. In particular, if $\{\mu_j^k\}_j$ is an orthonormal basis for the discrete spectrum of the weight $k$ Laplacian, then $\{\overline{\mu_j^{k}}\}_j$ is an orthonormal basis for the discrete spectrum in weight $-k$. These conjugated forms have Fourier expansions
\[
  \overline{\mu^{k}(z)}
  =
  \sum_{m \neq 0} \rho^{-k}(m)
  W_{\frac{-mk}{2\lvert m \rvert},it}(4\pi \lvert m \rvert y)e(mx),
\]
in which $\rho^{-k}(m) = \overline{\rho^{k}(-m)}$. In particular, we recognize that the second line of~\eqref{eq:discrete-holder-k2} can be written
\[
	\bigg(%
    \sum_{\vert t_j \vert \leq \vert s \vert^{1+\epsilon}}
    \frac{\vert \rho_{j,\mathfrak{a}}(1) \vert^2}{\cosh \pi t_j}
    \lvert L(w,\mu_{j,\mathfrak{a}}^1) \rvert^4
  \bigg)^\frac{1}{4}
  \bigg(%
    \sum_{\vert t_j \vert \leq \vert s \vert^{1+\epsilon}}
    \frac{\vert \rho_{j,\mathfrak{a}}(1) \vert^2}{\cosh \pi t_j}
		\lvert L(w,\mu_{j,\mathfrak{a}}^1 \times \chi) \rvert^4
  \bigg)^\frac{1}{4},
\]
in which each sum runs through an orthonormal basis of Maass forms of weight $1$ with spectral type $\vert t_j \vert \leq \vert s \vert^{1+\epsilon}$.

The Lindel\"{o}f hypothesis predicts that the $L$-functions in the line above grow slowly with respect to $\vert t_j \vert$ if $\Re w \geq \frac{1}{2}$. If we assume the Lindel\"{o}f hypothesis and apply the weight $1$ Kuznetsov--Proskurin trace formula, we obtain a conjectural bound for these moments, first presented in Conjecture~\ref{conj:intro-spectral-fourth-moment}. In abbreviated form, this states the following.

\begin{conjecture*}[Spectral Fourth Moment Conjecture for Weight $1$]
  Fix $r \in \mathbb{R}$ and any $\epsilon > 0$. As $T \to \infty$, we have
  \begin{equation*}
    \sum_{\vert t_j \vert \leq T}
    \frac{\vert \rho_j^{1}(1) \vert^2}
         {\cosh(\pi t_j)}
    \lvert L(\tfrac{1}{2}+ir,\mu_j^1) \rvert^4
    \ll_{N,r,\epsilon}
    T^{1+\epsilon},
  \end{equation*}
  where the sum runs through Maass forms of bounded type $t_j$ in an orthonormal basis of Maass
  forms of weight $1$ and level $N$.
\end{conjecture*}

The analogous conjecture is known for weight $0$ Maass forms~\cite{Iwaniec92, Motohashi92}, and we
expect that this conjecture is within current technology to prove.
Applying ideas from the proofs of~\cite{Iwaniec92, Motohashi92} leads to the
following
weaker bound, demonstrated in Appendix~\ref{app:huang_kuan} by Huang and Kuan.

\begin{proposition}[Huang and Kuan, Appendix~\ref{app:huang_kuan}]%
\label{prop:best_towards_conjecture}
  Fix $r \in \mathbb{R}$ and $\epsilon > 0$. As $T \to \infty$, we have
  \begin{equation*}
    \sum_{\vert t_j \vert \leq T}
    \frac{\vert \rho_j^{1}(1) \vert^2}
         {\cosh(\pi t_j)}
    \lvert L(\tfrac{1}{2}+ir,\mu_j^{1}) \rvert^4
    \ll_{N,r,\epsilon}
    T^{2+\epsilon},
  \end{equation*}
  where the sum runs through Maass forms of bounded type $t_j$ in an orthonormal basis of Maass
  forms of weight $1$ and level $N$.
\end{proposition}

Note that Conjecture~\ref{conj:intro-spectral-fourth-moment} and Proposition~\ref{prop:best_towards_conjecture} generalize trivially to the $L$-functions $L(w, \mu_{j,\mathfrak{a}}^{1})$ and $L(w,\mu_{j,\mathfrak{a}}^1 \times \chi)$, by raising the level $N$ and considering linear combinations of Maass forms.

We study the second line of~\eqref{eq:discrete-holder-k2} in two ways: assuming Conjecture~\ref{conj:intro-spectral-fourth-moment}
and using the unconditional result in Proposition~\ref{prop:best_towards_conjecture}.
Under the Conjecture, the second line of~\eqref{eq:discrete-holder-k2} is
$O_{h,w,\epsilon}(\lvert s \rvert^{\frac{1}{2}+\epsilon})$ on the line $\Re w = \frac{1}{2}$.
Under Proposition~\ref{prop:best_towards_conjecture}, the second line of~\eqref{eq:discrete-holder-k2} is
$O_{h,w,\epsilon}(\lvert s \rvert^{1+\epsilon})$ on the line $\Re w = \frac{1}{2}$.
On the line $\Re w = 1 + \epsilon$, absolute and uniform convergence of the $L$-functions and the basic
Kuznetsov bound~\eqref{eq:kuznetsov-k2} show that the second line is bounded (unconditionally) by
$O_{h,w,\epsilon}(\lvert s \rvert^{\frac{1}{2} + \epsilon})$.

We conclude that
\begin{equation*}
  \Sigma_{\mathrm{disc}}
  \ll_{h,w,\epsilon}
  \begin{cases}
    \lvert s \rvert^{1 + \epsilon} (1+ \vert s \vert^{\frac{1}{2}-\sigma}), & \Re w = \frac{1}{2},
    \\
    \lvert s \rvert^{2 + 2\epsilon} (1 + \vert s \vert^{\frac{1}{2} - \sigma}), & \Re w = 1 + \epsilon,
  \end{cases}
  \quad \text{under Conjecture~\ref{conj:intro-spectral-fourth-moment}}
\end{equation*}
and that
\begin{equation*}
  \Sigma_{\mathrm{disc}}
  \ll_{h,w,\epsilon}
	\begin{cases}
		\lvert s \rvert^{\frac{3}{2} + \epsilon}(1+ \vert s \vert^{\frac{1}{2}-\sigma}), & \Re w = \frac{1}{2},
		\\
		\lvert s \rvert^{2 + 2\epsilon}(1+ \vert s \vert^{\frac{1}{2}-\sigma}), & \Re w = 1+\epsilon,
	\end{cases}
  \qquad \text{unconditionally.}
\end{equation*}
Since $\Sigma_{\mathrm{disc}}$ is a meromorphic function of $w$
(cf. Remark~\ref{remark:discrete_k2_has_continuation_in_w}), the convexity principle can be used to
interpolate the growth of $\Sigma_{\mathrm{disc}}$ for $\frac{1}{2}<\Re w <1$.
So doing, we prove the following.

\begin{proposition}\label{prop:ds-growth-k2}
  Fix $w$ with $\Re w \in [\frac{1}{2}, 1)$.
  Assuming Conjecture~\ref{conj:intro-spectral-fourth-moment}, the  discrete spectrum
  $\Sigma_{\mathrm{disc}}$ of $D_h(s, w)$ satisfies
  \begin{equation}
    \Sigma_{\mathrm{disc}}(s,w)
    =
    O_{h, w, \epsilon}(\lvert s \rvert^{2\Re w + \epsilon}(1+\vert s \vert^{\frac{1}{2}-\sigma}))
  \end{equation}
  for any $\epsilon > 0$, as $\lvert \Im s \rvert \to \infty$ within a fixed vertical strip.
  Unconditionally,
  $\Sigma_{\mathrm{disc}}$ satisfies
  \begin{equation}
    \Sigma_{\mathrm{disc}}(s,w)
    =
    O_{h, w, \epsilon}(\lvert s \rvert^{\Re w + 1 + \epsilon}(1+ \vert s \vert^{\frac{1}{2}-\sigma})).
  \end{equation}
\end{proposition}

Combining the results of
Propositions~\ref{prop:continuous-growth},~\ref{prop:ns-growth}, and~\ref{prop:ds-growth-k2} and
exploiting symmetry under $w \mapsto 1-w$ completes our proof of Theorem~\ref{thm:Dh(s,w)-growth}.

\section{Bounds on Sharp Sums}\label{sec:sharp-cutoffs}

We are now ready to use the meromorphic continuation and polar behavior of $D_h(s, w)$ for $0 < \Re
w < 1$ to study the sums
\begin{equation*}
  S(X) =S(X;w,h)
  :=
  \sum_{m + h \leq X} r_2(m) \sigma_{1 - 2w}(m + h).
\end{equation*}
We will perform a standard, classical examination of smoothed versions of $S(X)$ to prove our
primary arithmetic theorem.

\begin{theorem}\label{thm:S(X)-growth}
  Fix $w$ with $0 < \Re w < 1$ and any $\epsilon > 0$.
  If $w \neq \frac{1}{2}$,
  \begin{align*}
    S(X)
    &=
    (4\pi)^{\frac{1}{2}}
    \zeta^*(2w) \varphi_h(\tfrac{1}{2} + w)
    \frac{X}{h^{w - \frac{1}{2}}}
    \\
    & \quad +
    (4\pi)^{\frac{1}{2}}
    \zeta^*(2-2w) \varphi_h(\tfrac{3}{2} - w)
    \frac{X^{2 - 2w}}
         {h^{\frac{1}{2}-w}(2 - 2w)}
    +
    \mathrm{Err}_w^2(X),
  \end{align*}
  where
  \begin{equation}
    \mathrm{Err}_w^2(X) \ll_{h, w, \epsilon}
    \begin{cases}
      X^{\frac{3}{2\Re w + 3} + \epsilon}
      \qquad & \Re w \in [\frac{1}{2}, 1),
      \\
      X^{2-2\Re w - \frac{2-2\Re w}{5-2 \Re w} + \epsilon}
      \qquad & \Re w \in (0, \frac{1}{2}].
    \end{cases}
  \end{equation}
  If $w=\frac{1}{2}$, we have instead
  \begin{align*}
&
  (4 \pi)^{\frac{1}{2}}
  \varphi_h(1)
  X \log X
  +
  (4 \pi)^{\frac{1}{2}}
  \varphi_h(1)(\gamma-\log(4\pi h))
  X
  \\
  & \qquad +
  (4 \pi)^{\frac{1}{2}} \varphi_h'(1)
  X
  -
  (4 \pi)^{\frac{1}{2}}
  \varphi_h(1)
  X
  +
  O_{h,w,\epsilon}\big(X^{\frac{3}{4}+\epsilon}\big).
\end{align*}
  Assuming Conjecture~\ref{conj:intro-spectral-fourth-moment}, these bounds improve to
  \begin{equation}
		\mathrm{Err}_w^2(X) \ll_{h, w, \epsilon}
    \begin{cases}
      X^{\frac{2\Re w + 1}{4 \Re w+1} + \epsilon}
      \qquad & \Re w \in [\frac{1}{2}, 1),
      \\
      X^{2 - 2 \Re w - \frac{2-2\Re w}{5-4 \Re w} + \epsilon}
      \qquad & \Re w \in (0, \frac{1}{2}],
    \end{cases}
  \end{equation}
  and the error term when $w = \frac{1}{2}$ improves to $O_{h, \epsilon}(X^{\frac{2}{3} +
  \epsilon})$.
\end{theorem}

\subsection*{Smooth weights}

To avoid some of the technical details inherent in applications of Perron's formula, we elect
to analyze the partial sum $S(X)$ using upper and lower bounds derived from smoothly-weighted
analogues of $S(X)$. In particular, we use two weight functions $u_{+y}(t)$ and $u_{-y}(t)$
which are smooth, non-increasing functions of compact support satisfying
\begin{equation*}
  u_{-y}(t) = \begin{cases}
    1 & t \leq 1 - \frac{1}{y}, \\
    0 & t \geq 1,
  \end{cases}
  \quad \text{and} \quad
  u_{+y}(t) = \begin{cases}
    1 & t \leq 1, \\
    0 & t \geq 1 + \frac{1}{y},
  \end{cases}
\end{equation*}
where $y > 1$ is an optimizing parameter we specialize below.
We note that these are the same smoothing functions the authors use in~\cite{hulse2020arithmetic},
and similar smoothed weighting functions are common in the literature.

Let $U_{-y}(s)$ and $U_{+y}(s)$ denote the Mellin transforms of $u_{-y}(t)$ and $u_{+y}(t)$,
respectively. These Mellin transforms satisfy
\begin{enumerate}
  \item $U_{\pm y}(s) = s^{-1} + O_s(1/y)$.
  \item $U'_{\pm y}(s) = -s^{-2} + O_s(1/y)$.
  \item For all $\alpha \geq 1$ and for $s$ constrained in a vertical strip with $\lvert s \rvert >
  \epsilon>0$, we have
  \begin{equation}\label{eq:U_alpha}
    U_{\pm y}(s) \ll_\epsilon \frac{1}{y} \Big( \frac{y}{1 + \lvert s \rvert} \Big)^\alpha.
  \end{equation}
\end{enumerate}
To study $S(X)$, we define the smoothed sums
\begin{align}\nonumber
  S_{\pm y}(X) &= S_{\pm y}(X;w,h)
  :=
  \sum_{m=0}^\infty r_{2}(m) \sigma_{1-2w}(m+h) u_{\pm y}\Big(\frac{m+h}{X}\Big)
  \\
  &=
  \frac{1}{2\pi i} \int_{(\sigma)}
  D_h(s-\tfrac{1}{2}+w,w) X^s U_{\pm y}(s) \, ds, \label{eq:Dh(s,w)-cutoff-integral}
\end{align}
where $\sigma > \max(1, 2 - 2\Re w)$ initially.
When $w$ is real, the summands in $S(X)$ and $S_{\pm y}(X)$ are non-negative and satisfy
$S_{-y}(X) \leq S(X) \leq S_{+y}(X)$.
We will study asymptotics for $S_{\pm y}(X)$, which imply asymptotics for $S(X)$ when $w$ is real.
We then generalize to $w$ non-real.

\subsection*{Bounding smoothed sums}

Beginning with the integral transforms~\eqref{eq:Dh(s,w)-cutoff-integral} for $S_{\pm y}(X)$,
shift the line of integration to $\sigma = 1 - \Re w + \epsilon$, for small $\epsilon$.
Bounded polynomial growth of $D_h(s - \frac{1}{2} + w, w)$ in $\lvert \Im s \rvert$ and arbitrary
polynomial decay of the weights $U_{\pm y}(s)$ shows that the integral converges absolutely.
By Cauchy's theorem, this shift introduces residues from poles from
$\Sigma_{\mathrm{reg}}$
at $s = 1$ and at $s = 2 - 2 w$ as detailed in Theorem~\ref{thm:D_h(s,w)-poles}.
The residues at these poles take a different form depending on if
$w = \frac{1}{2}$ or not.

If $w \neq \frac{1}{2}$, then the residues are
\begin{align*}
  &(4\pi)^{\frac{1}{2}}
  \zeta^*(2w) \varphi_h(\tfrac{1}{2}+w)
  \frac{X}{h^{w - \frac{1}{2}}} U_{\pm y}(1)
  \\
  &\quad +
  (4\pi)^{\frac{1}{2}}
  \zeta^*(2-2w) \varphi_h(\tfrac{3}{2}-w)
  \frac{X^{2 - 2w}}{h^{\tfrac{1}{2}-w}}
  U_{\pm y}(2 - 2w).
\end{align*}
Using that $U_{\pm y}(s) = s^{-1} + O_s(1/y)$, we write these residues as
\begin{align*}
  & (4\pi)^{\frac{1}{2}}
  \zeta^*(2-2w) \varphi_h(\tfrac{3}{2}-w)
  \frac{X^{2 - 2w}}{h^{\tfrac{1}{2}-w}(2-2w)}   \\
  &\quad +
  (4\pi)^{\frac{1}{2}}
  \zeta^*(2w) \varphi_h(\tfrac{1}{2}+w)
  \frac{X}{h^{w - \frac{1}{2}}}
  + O_w\Big(\frac{X}{y}   + \frac{X^{2-2\Re w}}{y}\Big)%
  .
\end{align*}

If $w = \frac{1}{2}$, the residue of the double pole at $s=1$ is instead
\begin{align*}
  &
  (4 \pi)^{\frac{1}{2}}
  \varphi_h(1)
  X \log X
  +
  (4 \pi)^{\frac{1}{2}}
  \varphi_h(1)(\gamma-\log(4\pi h))
  X
  \\
  & \qquad +
  (4 \pi)^{\frac{1}{2}}
  \varphi_h'(1)
  X
  -
  (4 \pi)^{\frac{1}{2}}
  \varphi_h(1)
  X
  +
  O_w\Big(\frac{X\log X}{y}\Big).
\end{align*}

To bound the contribution of the shifted contour integral, we address the
contribution of the terms $\Sigma_{\mathrm{reg}}$, $\Sigma_{\mathrm{cont}}$,
and $\Sigma_{\mathrm{disc}}$ separately. For
$\Sigma_{\mathrm{reg}}$ and $\Sigma_{\mathrm{cont}}$, we apply
Propositions~\ref{prop:ns-growth} and~\ref{prop:continuous-growth} to bound
growth by $O_{h,w,\epsilon}(\vert s \vert^{-1})$ and $O_{h,w,\epsilon}(\vert s
\vert^{-\frac{1}{6} + \frac{4}{3} \lvert \Re w - \frac{1}{2} \rvert + \epsilon})$, respectively. Choosing $\alpha = \max(1,\frac{5}{6} + \frac{4}{3} \vert \Re w - \frac{1}{2} \vert) + \epsilon$ in~\eqref{eq:U_alpha} guarantees that the contour integral
\[
	\frac{1}{2\pi i} \int_{(1-\Re w + \epsilon)}
		\Big(\Sigma_{\mathrm{cont}}(s-\tfrac{1}{2}+w,w)
			+ \Sigma_{\mathrm{reg}}(s-\tfrac{1}{2}+w,w) \Big) X^s U_{\pm y}(s) ds
\]
converges absolutely and satisfies the bound
\[
	O_{h,w,\epsilon}\Big( X^{1-\Re w + \epsilon} y^{\max(0,-\frac{1}{6} + \frac{4}{3} \vert \Re w - \frac{1}{2} \vert) + \epsilon}\Big).
\]

More care is needed to adequately address the contribution of the discrete spectrum $\Sigma_{\mathrm{disc}}$. For this, we shift the contour of integration farther left, to the line $\Re s = -\Re w + \epsilon$. This shift extracts infinitely many residues, at points of the form $s=1-w\pm it_j$, which total
\begin{align*}
  &  X^{1-w} \sum_{j,\mathfrak{a}}
		\frac{\Gamma(w+it_j)\Gamma(w-it_j)}
			{2^{2w-4}\pi^{2w - 1} h^{-\frac{1}{2}} c_\mathfrak{a}^{\frac{1}{2}-w}}
		b_{\mathfrak{a}}^w
		\overline{\rho^{-1}_{j,\mathfrak{a}}(-1)}\rho_{j,\infty}^{-1}(h)
    L(w, \overline{\mu_{j, \mathfrak{a}}^{-1}}) L(w, \overline{\mu_{j, \mathfrak{a}}^{-1}} \times \chi) \\
	& \,\, \Big(
	\frac{(X/h)^{it_j} \Gamma(2it_j) U_{\pm y}(1-w+it_j)}{\Gamma(w+it_j)\Gamma(1-w+it_j)}
	+ \frac{(X/h)^{-it_j} \Gamma(-2it_j) U_{\pm y}(1-w-it_j)}{\Gamma(w-it_j)\Gamma(1-w-it_j)}
	\Big).
\end{align*}
Stirling's approximation and the estimate $U_{\pm y}(s) \ll \frac{1}{y}(\frac{y}{1+ \vert s \vert})^\alpha$ imply that this residue sum is
\begin{align*}
	O_{h,w}\Big(
		\frac{X^{1-\Re w}}{y^{1-\alpha}} \sum_{j,\mathfrak{a}}
		\frac{\vert t_j \vert^{2\Re w -\frac{3}{2}-\alpha}}{\cosh \pi t_j}
	&	\vert \rho^{-1}_{j,\mathfrak{a}}(-1)\rho_{j,\infty}^{-1}(h)
    L(w, \overline{\mu_{j, \mathfrak{a}}^{-1}}) L(w, \overline{\mu_{j, \mathfrak{a}}^{-1}} \times \chi) \vert\Big).
\end{align*}
In this upper bound, the contribution of the dyadic sub-interval $\vert t_j \vert \sim T$ is
\begin{align*}
	& O_{h,w}\Big(
		X^{1-\Re w} y^{\alpha -1} T^{2\Re w - \frac{3}{2}-\alpha} \\
	&	\qquad \qquad \qquad
	\sum_{\mathfrak{a}} \sum_{\vert t_j \vert \sim T }
		\frac{\vert \rho^{-1}_{j,\mathfrak{a}}(-1)\rho_{j,\infty}^{-1}(h) \vert}{\cosh \pi t_j}
	\vert L(w, \overline{\mu_{j, \mathfrak{a}}^{-1}}) L(w, \overline{\mu_{j, \mathfrak{a}}^{-1}} \times \chi) \vert
	\Big) \\
	& \quad \ll_{h,w} \max_{\mathfrak{a}}
	X^{1-\Re w} y^{\alpha -1} T^{2\Re w - \frac{3}{2}-\alpha} \\
	&	\qquad \qquad
	\Big(\sum_{\vert t_j \vert \sim T}
		\frac{\vert \rho_{j,\infty}^{-1}(h) \vert^2}{\cosh \pi t_j}
	\Big)^{\frac{1}{2}}
	\Big(\sum_{\vert t_j \vert \sim T}
		\frac{\vert \rho^{1}_{j,\mathfrak{a}}(1)\vert^2}{\cosh \pi t_j}
	\vert L(w, \mu_{j, \mathfrak{a}}^{1}) L(w, \mu_{j, \mathfrak{a}}^{1} \times \chi) \vert^2
	\Big)^{\frac{1}{2}}  \\
	& \quad \ll_{h,w,\epsilon}
		X^{1-\Re w} y^{\alpha-1}  T^{2\Re w -\alpha} \\
	& \qquad \qquad \qquad \times
		\begin{cases}
		T^{\frac{1}{2}+\epsilon} (1+ T^{4(\frac{1}{2} - \Re w)}),
			 & \text{under Conjecture~\ref{conj:intro-spectral-fourth-moment}}, \\
		T^{\frac{1}{2}+\epsilon} (T^{1-\Re w} + T^{2-3\Re w}),
			 & \text{unconditionally,}
		\end{cases}
\end{align*}
following the convexity principle, the Kuznetsov trace formula, and our results on spectral fourth moments.
To guarantee convergence of the full $t_j$-sum, we choose $\alpha = \frac{3}{2} + 2 \vert \Re w - \frac{1}{2} \vert + \epsilon$ under Conjecture~\ref{conj:intro-spectral-fourth-moment} and choose $\alpha = 2 + \vert \Re w - \frac{1}{2} \vert + \epsilon$ unconditionally. It follows that the rightmost line of residues coming from $\Sigma_{\mathrm{disc}}$ is
\[
	\begin{cases}
			O_{h,w,\epsilon}\big(X^{1-\Re w} y^{\frac{1}{2}+2\vert \Re w - \frac{1}{2} \vert+ \epsilon}\big),
					 & \text{under Conjecture~\ref{conj:intro-spectral-fourth-moment}}, \\
			O_{h,w,\epsilon}\big(X^{1-\Re w} y^{1 + \vert \Re w - \frac{1}{2} \vert + \epsilon}\big),
				   & \text{unconditionally.}
	\end{cases}
\]

Lastly, we consider the contribution of the contour integral of the discrete spectrum $\Sigma_{\mathrm{disc}}$ on the left-shifted contour. Following Proposition~\ref{prop:ds-growth-k2}, we have
\[
	\Sigma_{\mathrm{disc}}(s-\tfrac{1}{2}+w,w)
		\ll_{h,w,\epsilon}
	\begin{cases}
		\vert s \vert^{2 \vert \Re w- \frac{1}{2} \vert + 2 + \epsilon},
		& \text{under Conjecture~\ref{conj:intro-spectral-fourth-moment}}, \\
		\vert s \vert^{\vert \Re w- \frac{1}{2} \vert + \frac{5}{2} + \epsilon},
		& \text{unconditionally,}
	\end{cases}
\]
on the line $\Re s = -\Re w + \epsilon$. Choosing $\alpha = 2 \lvert \Re w-
\frac{1}{2} \rvert + 3 + \epsilon$ under
Conjecture~\ref{conj:intro-spectral-fourth-moment} and $\alpha = \lvert \Re w-
\frac{1}{2} \rvert + \frac{7}{2} + \epsilon$ unconditionally in the bound $U_{\pm y}(s) \ll y^{\alpha -1} \vert s \vert^{-\alpha}$ produces
\begin{align*}
	& \frac{1}{2\pi i} \int_{(-\Re w + \epsilon)}
		\Sigma_{\mathrm{disc}}(s-\tfrac{1}{2}+w,w) X^s U_{\pm y}(s) ds \\
	& \qquad \ll_{h,w,\epsilon}
		\begin{cases}
			X^{-\Re w + \epsilon} y^{2 \vert \Re w- \frac{1}{2} \vert + 2 + \epsilon},
		& \text{under Conjecture~\ref{conj:intro-spectral-fourth-moment}}, \\
		  X^{-\Re w + \epsilon} y^{\vert \Re w- \frac{1}{2} \vert + \frac{5}{2} + \epsilon},
		& \text{unconditionally.}
		\end{cases}
\end{align*}

To balance the collected (unconditional) error terms
\begin{align*}
    \frac{X}{y} &+ \frac{X^{2-2\Re w}}{y}
    +
    X^{1-\Re w+\epsilon}
    y^{\max(0, -\frac{1}{6} + \frac{4}{3} \vert \Re w - \frac{1}{2} \vert) + \epsilon} \\
	&
	+	X^{1-\Re w} y^{1+ \vert \Re w - \frac{1}{2} \vert + \epsilon}
		+
		X^{-\Re w + \epsilon} y^{\vert \Re w- \frac{1}{2} \vert + \frac{5}{2} + \epsilon},
\end{align*}
we choose $y = X^\beta$ with
\[
	\beta =
	\begin{cases}
		\frac{2 \Re w}{2\Re w + 3}, &\Re w \in [\tfrac{1}{2},1), \\
		\frac{2 \Re w -2}{2\Re w -5}, & \Re w \in (0, \tfrac{1}{2}].
	\end{cases}
\]
If we assume Conjecture~\ref{conj:intro-spectral-fourth-moment}, the error terms are
\begin{align*}
    \frac{X}{y} &+ \frac{X^{2-2\Re w}}{y}
    +
    X^{1-\Re w+\epsilon}
    y^{\max(0, -\frac{1}{6} + \frac{4}{3} \vert \Re w - \frac{1}{2} \vert) + \epsilon} \\
	&
	+	X^{1-\Re w} y^{\frac{1}{2}+ 2\vert \Re w - \frac{1}{2} \vert + \epsilon}
		+
		X^{-\Re w + \epsilon} y^{2\vert \Re w- \frac{1}{2} \vert + 2 + \epsilon},
\end{align*}
which are balanced by choosing $y = X^\beta$ with
\[
	\beta =
	\begin{cases}
		\frac{2 \Re w}{4\Re w + 1}, &\Re w \in [\tfrac{1}{2},1), \\
		\frac{2 \Re w -2}{4\Re w -5}, & \Re w \in (0, \tfrac{1}{2}].
	\end{cases}
\]

Simplification with these choices of $y$ shows that both $S_{+ y}(X)$ and $S_{-y}(X)$ satisfy the asymptotic
relations for $S(X)$ claimed in the statement of Theorem~\ref{thm:S(X)-growth}, for any $w$ with
$\Re w \in (0, 1)$.
The bounds $S_{-y}(X) \leq S(X) \leq S_{\pm y}(X)$ (valid for $w$ real) then combine to prove
Theorem~\ref{thm:S(X)-growth} for all real $w$ in the interval $(0, 1)$.

For general $w$, we still have
\begin{align*}
  S(X;w,h) &= S_{+y}(X;w,h) + O\Big( \sum_{m = X-h}^{X+X/y-h} \lvert r_2(m) \sigma_{1-2w}(m+h) \rvert \Big) \\
  &=S_{\pm y}(X;w,h) + O\Big(S(X+\tfrac{X}{y};\Re w, h) - S(X; \Re w, h) \Big).
\end{align*}
Choosing $y=X^\beta$ as above and applying Theorem~\ref{thm:S(X)-growth} to bound the difference
$S(X+\tfrac{X}{y};\Re w, h) - S(X; \Re w, h)$ completes the proof for general $w$.

\begin{appendix}

\section{Inner Products of the Form \texorpdfstring{$\langle V_w, E_\mathfrak{a}^{-1}(\cdot, \overline{u})\rangle$}{<V,E(.,u)>}}%
\label{sec:V-E_a-inner-products}

This appendix contains the proof of Proposition~\ref{prop:V-E_a-inner-product}, which gives formulas
for inner products of the form $\langle V_w, E_\mathfrak{a}^{-1}(\cdot, \overline{u})\rangle$
for each cusp $\mathfrak{a}$ of $\Gamma_0(4)$.
We begin with a lemma which gives closed forms for the Dirichlet series $D_h(s,w)$ in the case
when $h = 0$.
We note that we slightly abuse notation and write $D_0(s, w)$ to mean the series $D_h(s, w)$ with $h
= 0$ \emph{and omitting the first term in the series} (which would have divided by $h = 0$).

\begin{lemma} Let $\chi = (\frac{-4}{\cdot})$. For $\Re s > 1 + \lvert \Re w - \frac{1}{2} \rvert$, we have
\begin{align*}
  D_0(s,w) &= \frac{4\zeta(s+\frac{1}{2}-w)\zeta(s-\frac{1}{2}+w)L(s+\frac{1}{2}-w,\chi)L(s-\frac{1}{2}+w,\chi)}{L(2s,\chi)}.
\end{align*}
\end{lemma}

\begin{proof}
  Comparing Euler products gives the verification.
\end{proof}

We now begin our proof of Proposition~\ref{prop:V-E_a-inner-product}, treating first the inner products
$\langle V_w, E_\mathfrak{a}^{-1}(\cdot, \overline{u})\rangle$ at the cusp $\mathfrak{a}=\infty$.

\begin{proposition}\label{prop:appendix-inner-product-infty} We have
\begin{align*}
  \big\langle V_w, E_\infty^{-1}(\cdot, \overline{u})\big\rangle
  &= \frac{2}{\sqrt{\pi}} \zeta^*(u+\tfrac{1}{2}-w)\zeta^*(u+w-\tfrac{1}{2})\\
  & \qquad \times \frac{L^*(u+\frac{1}{2}-w,\chi)L^*(u+w-\frac{1}{2},\chi)}{L^*(2u,\chi)}.
\end{align*}
\end{proposition}

\begin{proof}
The inner products $\langle V_w, E_\infty^{-1}(\cdot, \overline{u})\rangle$ may be understood by
unfolding the Eisenstein series.
To simplify our computations, we substitute
$\langle V_w, E_\infty^{-1}(\cdot, \overline{u})\rangle$ for
$\langle y^{1/2} \overline{\theta(z)}^2 E^*(z,w), E_\infty^{-1}(\cdot, \overline{u})\rangle$,
interpreting the latter via Gupta's generalization of the Zagier regularization method to congruence
subgroups~\cite{Gupta00, ZagierRankinSelberg}. (See also \S4.1 and Appendix A
in~\cite{HKLDWSphere}.)
We conclude that
\begin{align*}
  \big\langle V_w, E_\infty^{-1}(\cdot, \overline{u})\big\rangle
  &= \int_0^\infty \int_0^1 y^{u+\frac{1}{2}} \overline{\theta(z)}^2 E^*(z,w) \frac{dxdy}{y^2} \\
  &= 2\sum_{n =1}^\infty \frac{r_2(n) \sigma_{1-2w}(n) n^{w-\frac{1}{2}}}{(2\pi n)^{u}} \int_0^\infty y^{u} K_{w-\frac{1}{2}}(y) e^{-y} \frac{dy}{y} \\
  &= \frac{\Gamma(u+\frac{1}{2}-w)\Gamma(u-\frac{1}{2}+w)}{(4\pi)^{u-\frac{1}{2}} \Gamma(u+\frac{1}{2})} \sum_{n \geq 1} \frac{r_2(n) \sigma_{1-2w}(n)}{n^{u+\frac{1}{2}-w}}.
\end{align*}
The Dirichlet series which remains is $D_0(u, w)$, which we studied above.
To simplify, we rewrite $\zeta(s)$ and $L(s,\chi)$ in terms of their completions $\zeta^*(s) =
\pi^{-s/2}\Gamma(\frac{s}{2})\zeta(s)$ and $L^*(s,\chi) = (\frac{\pi}{4})^{-s/2}
\Gamma(\frac{s+1}{2})L(s,\chi)$ and apply the Gauss duplication formula to the gamma factors which
remain.
\end{proof}

The inner products involving the Eisenstein series at the other cusps can be computed by changing
variables and modifying techniques from the $\mathfrak{a} = \infty$ case. We first discuss
$\mathfrak{a}=0$.

\begin{proposition}\label{prop:appendix-inner-product-0}
  We have
  \begin{align*}
  \big\langle V_w, E_0^{-1}(\cdot, \overline{u})\big\rangle
  &= \frac{e(\frac{-1}{4})}{2^{2u-1}\sqrt{\pi}}\zeta^*(u+\tfrac{1}{2}-w)\zeta^*(u+w-\tfrac{1}{2})\\
  & \qquad \times \frac{L^*(u+\frac{1}{2}-w,\chi)L^*(u+w-\frac{1}{2},\chi)}{L^*(2u,\chi)}.
\end{align*}
\end{proposition}

\begin{proof}
Let $d\mu$ denote the Haar measure on $\Gamma_0(4) \backslash \mathcal{H}$, normalized so that $\Gamma_0(4) \backslash \mathcal{H}$ has area $1$. Zagier regularization gives
\[
  \langle V_w, E_0^{-1}(\cdot, \overline{u}) \rangle
  =
  \iint_{\Gamma_0(4) \backslash \mathcal{H}}
  \Im(z)^\frac{1}{2} \overline{\theta(z)}^2 E^*(z,w)
  \overline{E_0^{-1}(z,\overline{u})} d\mu.
\]
As $E_0^{-1}(z,w) = (z/\vert z \vert) E_\infty^{-1}(\sigma_0 z, w)$ with $\sigma_0 =
(\begin{smallmatrix} 0 & -1 \\ 4 & 0 \end{smallmatrix})$, the change of variables $z \mapsto
\sigma_0^{-1} z$ rewrites the inner product as
\begin{align*}
 \iint_{\sigma_0(\Gamma_0(4) \backslash \mathcal{H})}
 \Im(\sigma_0^{-1} & z)^\frac{1}{2}
\overline{\theta(\sigma_0^{-1} z)}^2 E^*(\sigma_0^{-1} z,w)
\Big(\frac{\vert \sigma_0^{-1} z \vert}{\sigma_0^{-1} z} \Big)
\overline{E_\infty^{-1}(z,\overline{u})} d\mu \\
= & e(\tfrac{-1}{4}) \iint_{\Gamma_0(4) \backslash \mathcal{H}}
\Im(z)^{\frac{1}{2}}
\overline{\theta(z)}^2 E^*(4z,w)
\overline{E_\infty^{-1}(z,\overline{u})} d\mu,
\end{align*}
by the functional equations of $\theta(z)$ and $E^*(z,w)$, careful treatment of the square roots,
and the fact that $\sigma_0$ maps $\Gamma_0(4) \backslash \mathcal{H}$ to itself.
A standard unfolding shows
\begin{align*}
  \langle V_w, E_0^{-1}(\cdot, \overline{u}) \rangle
   &= e(\tfrac{-1}{4}) \int_0^\infty \int_0^1
   y^{u+\frac{1}{2}} \overline{\theta(z)}^2 E^*(4z,w) \frac{dxdy}{y^2} \\
  &= e(\tfrac{-1}{4}) \frac{\Gamma(u+\frac{1}{2}-w)\Gamma(u-\frac{1}{2}+w)}{2(16\pi)^{u-\frac{1}{2}} \Gamma(u+\frac{1}{2})}
  \sum_{n \geq 1} \frac{r_2(4n) \sigma_{1-2w}(n)}{n^{u+\frac{1}{2}-w}}.
\end{align*}
As $r_2(4n) = r_2(n)$, we complete the proof as in Proposition~\ref{prop:appendix-inner-product-infty}.
\end{proof}

Finally, we show that the inner product $\langle V_w, E_\mathfrak{a}^{-1}(\cdot, \overline{u})
\rangle$ vanishes in the case $\mathfrak{a}= \frac{1}{2}$.

\begin{proposition}
We have
\[
  \big\langle V_w, E_\frac{1}{2}^{-1}(\cdot, \overline{u})\big\rangle = 0.
\]
\end{proposition}

\begin{proof} Since $E_{1/2}^{-1}(z,w) = (\frac{2z+1}{\vert 2z+1 \vert}) E_\infty^{-1}(\sigma_{1/2} z, w)$ with $\sigma_{1/2} = (\begin{smallmatrix} 1 & 0 \\ 2 & 1 \end{smallmatrix})$,
a change of variables $z \mapsto \sigma_{1/2}^{-1} z$ along the lines of Proposition~\ref{prop:appendix-inner-product-0} produces
\[
  \big\langle V_w, E_\frac{1}{2}^{-1}(\cdot, \overline{u})\big\rangle
  = \iint_{\Gamma_0(4) \backslash \mathcal{H}}
  \Im(z)^{\frac{1}{2}}
  \big(\overline{\theta(\tfrac{z}{4})}-\overline{\theta(z)}\big)^2 E^*(z,w) \overline{E_\infty^{-1}(z,\overline{u})} d\mu,
\]
in which we've simplified by applying~\eqref{eq:theta-1/2-cusp-transformation}, $\mathrm{SL}_2(\mathbb{Z})$-invariance of $E^*(z,w)$, and that $\sigma_{1/2}(\Gamma_0(4) \backslash \mathcal{H}) = \Gamma_0(4) \backslash \mathcal{H}$.
At this point, unfolding gives
\begin{equation} \label{eq:1/2-cusp-unfolded}
  \big\langle V_w, E_\frac{1}{2}^{-1}(\cdot, \overline{u})\big\rangle
  = \int_0^\infty \int_0^1 y^{u+\frac{1}{2}}
  \big(\overline{\theta(\tfrac{z}{4})}-\overline{\theta(z)}\big)^2 E^*(z,w) \frac{dxdy}{y^2}.
\end{equation}

The Fourier expansion of $E^*(z,w)$ with respect to $x=\Re z$ is supported on phases of the form
$e(mx)$, with $m$ integral. On the other hand, the Fourier expansions of
$\theta(\frac{z}{4})-\theta(z)$ and $(\theta(\frac{z}{4})-\theta(z))^2$ are supported on
non-integral phases, as seen in~\eqref{eq:r1o-expansion} and~\eqref{eq:r2o-expansion}, respectively.
Thus the $x$-integral in~\eqref{eq:1/2-cusp-unfolded}, which extracts the constant Fourier
coefficient of the product, vanishes.
\end{proof}

\clearpage{}

\section{Relating \texorpdfstring{$E(z,w)$}{E(z,w)} to Eisenstein Series of Level \texorpdfstring{$4$}{4}}%
\label{sec:Eisenstein-decomposition}

This appendix proves the identity
\[E(z,w) = E_\infty(z,w) + 4^w E_0(z,w) + E_{1/2}(z,w)\]
presented in Proposition~\ref{prop:Elevel1_is_Elevel4_sum}. To begin, we note that the quotient $\Gamma_0(4) \backslash \Gamma_0(1)$ may be represented by the six matrices
\[\left(\begin{matrix} 1 & 0 \\ 0 & 1 \end{matrix} \right),
\left(\begin{matrix} 1 & 0 \\ 1 & 1 \end{matrix} \right),
\left(\begin{matrix} 1 & 0 \\ 3 & 1 \end{matrix} \right),
\left(\begin{matrix} 1 & -1 \\ 0 & 1 \end{matrix} \right),
\left(\begin{matrix} 1 & 1 \\ 1 & 2 \end{matrix} \right),
\left(\begin{matrix} 1 & 0 \\ 2 & 1 \end{matrix} \right),\]
which we denote by $\gamma_1,\ldots,\gamma_6$, respectively.  It follows that
\[E(z,w) = \sum_{i=1}^6 \sum_{\gamma \in \Gamma_\infty \backslash \Gamma_0(4)} (\Im \gamma \gamma_i z)^w = \sum_{i=1}^6 E_\infty(\gamma_i z,w).\]

The contribution towards $E(z,w)$ from $i=1$ is exactly $E_\infty(z,w)$.  Likewise, the contribution towards $E(z,w)$ from $i=6$ is exactly $E_{1/2}(z,w)$.  For the remaining $i$, we compute
\begin{align*}
E_\infty(\gamma_2 z, w) &= \frac{1}{2} \sum_{\substack{ c, d\in \mathbb{Z} \\ (4c,d)=1}} \frac{y^w}{\vert (4c+d) z+ d \vert^{2w}}, \\
E_\infty(\gamma_3 z, w) &=  \frac{1}{2} \sum_{\substack{ c, d\in \mathbb{Z} \\ (4c,d)=1}} \frac{y^w}{\vert (4c+3d) z+ d \vert^{2w}}, \\
E_\infty(\gamma_4 z, w) &=  \frac{1}{2} \sum_{\substack{ c, d\in \mathbb{Z} \\ (4c,d)=1}} \frac{y^w}{\vert (4c+d) z- 4c \vert^{2w}}, \\
E_\infty(\gamma_5 z, w) &=  \frac{1}{2} \sum_{\substack{ c, d\in \mathbb{Z} \\ (4c,d)=1}} \frac{y^w}{\vert (4c+d) z+ (4c+2d) \vert^{2w}}.
\end{align*}

In each of the four sums at right, the denominators may be written in the form $A z +B$, in which $A$ is odd and $(A,B)=1$.  Within this larger parameter space, the sum for $E_\infty(\gamma_2 z, w)$ exhausts the pairs with $A \equiv B \bmod 4$.  Likewise, the $\gamma_3$ sum exhausts pairs with $A \equiv - B \bmod 4$, the $\gamma_4$ sum exhausts pairs with $B \equiv 0 \bmod 4$, and the $\gamma_5$ sum exhausts pairs with $B \equiv 2 \bmod 4$. We conclude that
\[\sum_{i=2}^5 E_\infty(\gamma_i z,w) = \frac{1}{2} \sum_{ \substack{ A,B \in \mathbb{Z} \\ (A,2B)=1}} \frac{y^w}{\vert A z + B \vert^{2w}} = 4^w E_0(z,w),\]
which completes the proof of Proposition~\ref{prop:Elevel1_is_Elevel4_sum}.

\clearpage{}

\section{Weighted Fourth Moment of $L$-functions associated to weight 1 Maass forms}\label{app:huang_kuan}
\begin{center}
  Tinghao Huang, Chan Ieong Kuan
\end{center}

\begin{abstract}
  Adapting Iwaniec's methods of investigating weighted fourth moments of
  $L$-functions associated to weight $0$ Maass forms, we establish a similar
  bound for weight $1$ Maass forms.
\end{abstract}

\vspace{10mm}

The purpose of this short note is to establish a weighted fourth moment of
$L$-functions associated to weight 1 Maass forms.
We use an analogous approach as in~\cite[Theorem~1]{Iwaniec79}, except adapted
for weight $1$.

A brief summary of the approach would be as follows: we start with the trace formula of weight 1, given as in \cite{DFIsub}. Using a result from Humphries \cite{Humphries2016} to estimate the resulting integrals, one would obtain a version of large sieve inequality. The fourth moment can then be treated by estimating the square of $L$-functions via approximate functional equation, and applying large sieve inequality to the resulting expression.

\subsection{Preliminaries}
Fix an orthonormal basis of Maass eigenforms $\{ \mu_j(z) \}$ spanning the discrete spectrum of the Laplacian of weight $1$, where $\mu_j$ has type $\tfrac{1}{2} + it_j$. These are of level $\Gamma_0(M)$, having nebentypus $\chi$, and have Fourier--Whittaker expansions
\[ \mu_j(x+iy) = \sum_{n \in \mathbb{Z}} c_j(n,y) e(nx), \]
where
\[
	c_j(n,y) = \begin{cases}
		\rho_j(n) W_{\frac{n}{2|n|}, it_j} (4\pi |n|y), &\text{if } n \neq 0; \\
		\rho_j(0) y^{\frac{1}{2}+it_j}, &\text{if } n = 0,\, t_j \in i\mathbb{R}; \\
		0, &\text{if } n=0, t_j \in \mathbb{R}.
		\end{cases}
\]
We also write $\rho_j(n) = \rho_j(1) \lambda_j(n) |n|^{-1/2}$ for $n \neq 0$. As for the continuous spectrum, let $E_\mathfrak{a}^1(z,s)$ denote the weight $1$ Eisenstein series associated to the cusp $\mathfrak{a}$, which has Fourier expansion
\[ E_\mathfrak{a}^1(z,s) = \sum_{n \in \mathbb{Z}} \rho_\mathfrak{a}(n,y,s) e(nx), \]
where $\rho_\mathfrak{a}(n,y,s) = \rho_\mathfrak{a}(n,s) W_{\frac{n}{2|n|}, s-\frac{1}{2}} (4\pi |n|y)$ if $n \neq 0$.

We quote the following trace formula from \cite{DFIsub} here:
\begin{proposition}[Duke-Friedlander-Iwaniec] \label{propo:DFI}
  For any positive integers $m,n$ and any real number $r$, we have
  \begin{align*}
    &\sum_j \frac{\overline{\rho_j(m)} \rho_j(n)}{\cosh \pi(r-t_j) \cosh \pi(r+t_j)} + \sum_{\mathfrak{a}} \frac{1}{4\pi} \int_{\mathbb{R}} \frac{\overline{\rho_\mathfrak{a}(m,\tfrac{1}{2}+it)} \rho_\mathfrak{a}(n,\tfrac{1}{2}+it)}{\cosh \pi(r-t) \cosh \pi(r+t)} \,dt \\
    & \quad
		= \frac{|\Gamma(\frac{1}{2}+ir)|^2}{4\pi^3 \sqrt{mn}} \bigg( \delta_{m=n} -8\pi \sqrt{mn} \sum_{c \equiv 0 (M)} \frac{S_\chi(m,n;c)}{c^2} \! \int_{-i}^i \! K_{2ir}(\tfrac{4\pi \sqrt{mn}}{c}\zeta ) \,d\zeta \bigg),
  \end{align*}
  where the $\zeta$-integral runs counter-clockwise along the right half of the unit circle and $S_\chi(m,n;c)$ is the twisted Kloosterman sum.
\end{proposition}

We will also require the following integral estimate from Humphries \cite{Humphries2016}:
\begin{lemma}[Humphries] \label{lemma:humph0T}
  For $T > 0$, we have the bound
  \[ \int_0^T -2at \int_{-i}^i K_{2it}(a\zeta) \,d\zeta \,dt \ll \begin{cases} \sqrt{a} & \text{if $a \geq 1$,} \\ a(1+\log (1/a)) & \text{if $0<a<1$} \end{cases} \]
uniformly in $T$, and the $\zeta$-integral is the same as the previous proposition.
\end{lemma}

The following corollary follows easily from this lemma.
\begin{lemma}\label{lemma:humph}
  For $T > 0$, we have the bound
  \[ \Phi(a) := \int_{T/4}^{2T} -2at \int_{-i}^i K_{2it}(a\zeta) \,d\zeta \,dt \ll \begin{cases} \sqrt{a} & \text{if $a \geq 1$,} \\ a(1+\log (1/a)) & \text{if $0<a<1$} \end{cases} \]
uniformly in $T$, and the $\zeta$-integral is the same as the lemma above.
\end{lemma}

Lastly, we quote the following estimate of the size of coefficients of Maass forms from \cite{DFIsub}.
\begin{proposition}[Lemma~19.3 of Duke-Friedlander-Iwaniec] \label{propo:DFIest}
  With the notations in this appendix, we have
  \[ \sum_{1 \leq n \leq N} n |\rho_j(n)|^2 \ll ( \tfrac{N}{M}+1) |t_j| e^{\pi |t_j|}. \]
\end{proposition}

\subsection{Results}
For a sequence of complex numbers $\alpha = (a_n)_{1 \leq n \leq N}$, we define
\[ L_j(\alpha) := \sum_{\frac{1}{2}N \leq n \leq N} a_n \lambda_j(n). \]

We have the following version of weighted large sieve inequality,
\begin{theorem} \label{thm:wt_1_large_sieve}
With the notations in this appendix, we have
\[
	\sum_{\frac{1}{2} T \leq \vert t_j \vert \leq T}
		\frac{|\rho_j(1)|^2}{\cosh(\pi t_j)} |L_j(\alpha)|^2
	\ll_{\epsilon} \bigg( N + T + \frac{N^{3/2}}{T} \bigg) (NT)^{\epsilon} \| \alpha \|^2,
\]
where $\| \alpha \|$ is the $\mathcal{L}^2$-norm of the finite sequence $\alpha$.
\end{theorem}

As a nontrivial corollary, we derive the following weighted fourth moment estimate.
\begin{theorem} \label{thm:wt_1_4th_moment}
  With the previous notations, for $\Re s = \frac{1}{2}$, we have
  \[ \sum_{|t_j| \leq T} \frac{|\rho_j(1)|^2}{\cosh(\pi t_j)} |L(s,\mu_j)|^4 \ll_{\epsilon}
		( T^2 + |s|^3) (|s|T)^{\epsilon}. \]
\end{theorem}

\subsection{Proof of the large sieve inequality}
To start, we rewrite the sum as follows,
\begin{align} \label{eq:large-sieve-lhs}
	\sum_{\frac{T}{2} \leq \vert t_j\vert \leq T} \!\!
		\frac{|\rho_j(1)|^2}{\cosh (\pi t_j)} \vert L_j(\alpha) \vert^2
		= \!\!\! \sum_{\frac{T}{2} \leq \vert t_j \vert \leq T} \!\!
		\frac{1}{\cosh (\pi t_j)}
		\bigg\vert \sum_{\frac{N}{2} \leq n \leq N} \!\!\!\! \sqrt{n} a_n \rho_j(n) \bigg\vert^2. \,\,\,\,
\end{align}

Without loss of generality, we assume $N \gg 1$ and $T > N^{\epsilon}$. While the former is obvious, the latter requires a short explanation. For $T \leq N^{\epsilon}$, with the aid of Proposition \ref{propo:DFIest}, we have
\begin{align*}
  & \sum_{|t_j| \leq T} \frac{1}{\cosh (\pi t_j)}
		\bigg\vert \sum_{\frac{1}{2}N \leq n \leq N} \sqrt{n} a_n \rho_j(n) \bigg\vert^2
	 \leq
		\sum_{|t_j| \leq T} \frac{\Vert \alpha \Vert^2}{\cosh(\pi t_j)}
		\sum_{\frac{1}{2}N \leq n \leq N} n|\rho_j(n)|^2 \\
	& \qquad \ll
		\sum_{|t_j| \leq T} \Vert \alpha \Vert^2 \cdot \left( \tfrac{N}{M}+1 \right) |t_j|
		\ll N^{1+3\epsilon} \| \alpha \|^2,
\end{align*}
which is acceptably small.

Note that for $|t_j| \in [\frac{T}{2},T]$, we have
\[ \frac{1}{\cosh (\pi t_j)} \ll \int_{\frac{T}{4}}^{2T} \frac{dt}{\cosh(\pi (t-t_j)) \cosh(\pi (t+t_j))\Gamma(\frac{1}{2}+it) \Gamma(\frac{1}{2}-it)}. \]
It follows that~\eqref{eq:large-sieve-lhs} is bounded above by
\begin{align*}
  S(T,N) &:=
		\frac{1}{T} \sum_j \int_{\frac{T}{4}}^{2T}
		\frac{t \,dt}{\cosh(\pi (t-t_j)) \cosh(\pi (t+t_j))\Gamma(\frac{1}{2}+it) \Gamma(\frac{1}{2}-it)} \\
		&\qquad \times
		\bigg\vert \sum_{\frac{1}{2}N \leq n \leq N}  \sqrt{n} a_n \rho_j(n) \bigg\vert^2.
\end{align*}

Expanding the square, we can apply Proskurin's Kuznetsov formula (as in Proposition \ref{propo:DFI}) and positivity to obtain
\begin{align} \label{eq:P-K-off-diagonal}
  S(T,N) &\ll
		T^{-1}\sum_{\frac{N}{2} \leq m \leq N} |a_m|^2 \int_{\frac{T}{4}}^{2T} \frac{ t}{4\pi^3} \,dt \\
		&\quad
		+ T^{-1} \sum_{\frac{N}{2} \leq m,n \leq N} \overline{a_m} a_n
			\sum_{c \equiv 0 (M)} \frac{S_\chi(m,n;c)}{c}
			\Phi \left(\frac{4\pi \sqrt{mn}}{c}\right).
\end{align}
The first term of the right-hand side of the inequality is clearly $O(T \| \alpha \|^2)$.

As for the second term, we split the $c$-sum relative to the size of $4\pi \sqrt{mn}$.
For convenience, we define the following sums:
\begin{align*}
  S_1(T, N) :&=  \sum_{\frac{N}{2} \leq m,n \leq N} \overline{a_m} a_n \sum_{\substack{c \leq 4\pi \sqrt{mn} \\ c \equiv 0 (M)}} \frac{S_\chi(m,n;c)}{c} \Phi \left(\frac{4\pi \sqrt{mn}}{c}\right) \\
  S_2(T, N) :&=  \sum_{\frac{N}{2} \leq m,n \leq N} \overline{a_m} a_n \sum_{\substack{c > 4\pi \sqrt{mn} \\ c \equiv 0 (M)}} \frac{S_\chi(m,n;c)}{c} \Phi \left(\frac{4\pi \sqrt{mn}}{c}\right)
\end{align*}
To understand the size of the second term at right in~\eqref{eq:P-K-off-diagonal}, it suffices to estimate the two sums above.

\begin{proposition}
  With the notations as above, both $S_1(T,N)$ and $S_2(T,N)$ are bounded above by $N^{3/2 + \epsilon} \| \alpha \|^2$ for any $\epsilon > 0$.
\end{proposition}

\begin{proof}
  For $S_1(T,N)$, we note that the argument in $\Phi$ is bounded below by 1. Using the appropriate bound from Lemma \ref{lemma:humph} and the Cauchy--Schwartz inequality, we calculate that
  \begin{align*}
    S_1(T,N) &\ll \sum_{\frac{N}{2} \leq m,n \leq N} |a_m||a_n| \sum_{\substack{c \leq 4\pi \sqrt{mn} \\ c \equiv 0 (M)}} \frac{|S_\chi(m,n;c)|}{c} \left(\frac{4\pi \sqrt{mn}}{c}\right)^{1/2} \\
&\ll N^{\frac{3}{2} + \epsilon} \| \alpha \|^2.
  \end{align*}
  Similarly, for $S_2(T,N)$, we have
  \begin{align*}
    S_2(T,N) &\ll N^{\epsilon} \sum_{\frac{N}{2} \leq m,n \leq N} |a_m||a_n| \sum_{\substack{c > 4\pi \sqrt{mn} \\ c \equiv 0 (M)}} \frac{|S_\chi(m,n;c)|}{c} \left(\frac{4\pi \sqrt{mn}}{c}\right) \\
&\ll N^{\frac{3}{2} + \epsilon} \| \alpha \|^2.
\qedhere
  \end{align*}
\end{proof}

Putting the estimate of the first and second term together, we immediately obtain the large sieve inequality.

\subsection{Proof of Theorem~\ref{thm:wt_1_4th_moment}}
For easier presentation of the methodology, we assume the nebentypus $\chi$ is a primitive character, avoiding any oldforms.

With Hecke relations, we have
\[
	L(s,\mu_j)^2
		= \sum_{n \geq 1} \frac{c_n}{n^s},
	\text{ where }
		c_n = \sum_{md^2 = n} \chi(d) \sigma_0(m) \lambda_j(m).
\]

As the conductor of $L(s,\mu_j)^2$ is of size $N \asymp |s-it_j|^2 |s+it_j|^2$ in the $s$- and $t_j$-aspects, by approximate functional equation and dyadic partition, one can see that
\[
   \sum_{|t_j| \leq T} \frac{|\rho_j(1)|^2}{\cosh (\pi t_j)}
			\left| L(s,\mu_j) \right|^4
	\ll (NT)^{\epsilon} \!\! \max_{X \ll N^{\frac{1}{2}+\epsilon}}
		\sum_{|t_j| \leq T} \frac{|\rho_j(1)|^2}{\cosh (\pi t_j)} \frac{|\sum_n c_n V(\tfrac{n}{X})|^2}{X},
\]
where $V(x)$ is a smooth function with support inside $[1,2]$. To prove our result, we therefore consider estimates for
\[
	\Sigma(T,X)
		:= \sum_{\frac{T}{2} \leq |t_j| \leq T} \frac{|\rho_j(1)|^2}{\cosh (\pi t_j)}
		\bigg\vert
			\sum_{m \leq X} \lambda_j(m) \sigma_0(m) \sum_{d \leq \sqrt{X/m}} \chi(d)
		\bigg\vert^2.
\]
We can estimate $\Sigma(T,X)$ using the large sieve inequality with sieve weights $a_m = \sigma_0(m) \sum_{d \leq \sqrt{X/m}} \chi(d)$.  Since $\sigma_0(m) \ll m^\epsilon$, a trivial estimate of the $d$-sum provides $a_m \ll X^{1/2} m^{-1/2+\epsilon}$. Thus $\Vert \alpha \Vert^2 \ll X^{1+\epsilon}$. The large sieve inequality of Theorem~\ref{thm:wt_1_large_sieve} then implies
\[
	\Sigma(T,X) \ll (X + T + X^{\frac{3}{2}} T^{-1}) (XT)^{\epsilon} X^{1+\epsilon}.
\]
Dividing by $X$ and taking the maximum over $X \ll N^{1/2+\epsilon}$ gives
\[
	\sum_{\frac{T}{2} \leq |t_j| \leq T}
		\frac{|\rho_j(1)|^2}{\cosh (\pi t_j)} \left| L(s,\mu_j) \right|^4
	\ll (N^\frac{1}{2} + T + N^{\frac{3}{4}} T^{-1}) (N T)^\epsilon.
\]

Since $N \ll |s-iT|^2 |s+iT|^2 \ll \vert s \vert^4 + T^4$, we have
\[
	\sum_{\frac{T}{2} \leq |t_j| \leq T}
		\frac{|\rho_j(1)|^2}{\cosh(\pi t_j)} |L(s,\mu_j)|^4 \ll_{\epsilon}
	( T^2 + \vert s \vert^3 T^{-1}) (|s|T)^{\epsilon}.
\]
Summing dyadic intervals up to $T$ completes the proof of Theorem \ref{thm:wt_1_4th_moment}.
\end{appendix}

\vspace{20 mm}
\bibliographystyle{alpha}
\bibliography{compiled_bibliography.bib}

\end{document}